\documentclass[twoside,11pt]{article}

\usepackage{jmlr2e}

\usepackage{graphicx}
\usepackage{amsmath}
\usepackage{amsfonts}
\usepackage{amssymb}
\usepackage{mathrsfs}
\usepackage{enumerate}
\usepackage{color}
\usepackage{xcolor}
\usepackage{enumitem}
\usepackage{hyperref}
\hypersetup{
    colorlinks=true, 
    linktoc=all,     
    linkcolor=blue,  
}

\newcommand{\udef}{\mathrel{\mathop:}=}

\newcommand{\dd}{{\rm d}}
\newcommand{\N}{\mathbb{N}}
\newcommand{\de}{\mathrm{d}}
\newcommand{\partialc}{\partial^c} 
\newcommand{\parder}[2]{\frac{\partial#1}{\partial#2}}

\newcommand{\abs}[1]{\lvert#1\rvert}

\DeclareMathOperator{\cl}{cl}
\newcommand{\rev}[1]{{\color{black}#1 }}

\usepackage{lastpage}
\jmlrheading{23}{2022}{1-\pageref{LastPage}}{1/23}{}{23-0000}{F.V. Difonzo, V. Kungurtsev and J. Mare{\v c}ek}

\ShortHeadings{Stochastic Langevin Differential Inclusions}{F.V. Difonzo, V. Kungurtsev and J. Mare{\v c}ek}
\firstpageno{1}

\begin{document}

\title{Stochastic Langevin Differential Inclusions with Applications to Machine Learning}

\author{Fabio V. Difonzo, Vyacheslav Kungurtsev, and Jakub Mare{\v c}ek}

\editor{My editor}

\begin{center}
\maketitle
\end{center}

\begin{abstract}
Stochastic differential equations of Langevin-diffusion form have received significant attention, thanks to their foundational role in both Bayesian sampling algorithms and optimization in machine learning. In the latter, they serve as a conceptual model of the stochastic gradient flow in training over-parameterized models. However, the literature typically assumes smoothness of the potential, whose gradient is the drift term. Nevertheless, there are many problems for which the potential function is not continuously differentiable, and hence the drift is not Lipschitz continuous everywhere. This is exemplified by robust losses and Rectified Linear Units in regression problems. In this paper, we show some foundational results regarding the flow and asymptotic properties of Langevin-type Stochastic Differential Inclusions under assumptions appropriate to the machine-learning settings. In particular, we show strong existence of the solution, as well as an asymptotic minimization of the canonical free-energy functional.
\end{abstract}

\section{Introduction}\label{sec:intro}


In this paper, we study the following stochastic differential inclusion,
\begin{equation}\label{eq:langinc}
\de X_t \in -F(X_t) \de t+ \sqrt{2\sigma}\,\de B_t
\end{equation}
wherein $F(x):\mathbb{R}^n\rightrightarrows\mathbb{R}^n$ is a set-valued map. We are particularly interested in the case where
$F(x)$ is the Clarke subdifferential of some continuous tame function $f(x)$. This is motivated by the recent interest in studying Langevin-type diffusions in the context of machine learning applications, both as a scheme for sampling in a Bayesian framework (as spurred by the seminal work~\cite{welling2011bayesian}) and as a model of the trajectory of stochastic gradient descent, with a view to understanding the asymptotic properties of training deep neural networks~\cite{hu2017diffusion}. It is typically assumed that $F(x)$ above is Lipschitz and as such its potential $f(x)$ is continuously differentiable. In many problems of relevance, including empirical risk minimization with robust loss (e.g., $l1$ or Huber) and neural networks with ReLU activations, this is not the case, and yet there is at least partial empirical evidence suggesting that the long-term behavior of a numerically similar operation is similar in its capacity to generate a stochastic process which minimizes a Free Energy associated with the learning problem.

This paper is organized as follows. In Section~\ref{sec:exceptionalset} we study the functional analytical properties of $F(x)$ as it appears in~\eqref{eq:langinc} when it represents a noisy estimate of a subgradient element of an empirical loss function that itself satisfies the conditions of a \emph{definable} potential, especially as it appears in the context of deep learning applications. Note that the research program undertaken relates to a recent conjecture of Bolte and Pauwels \cite[Remark 12]{bolte2021conservative} that suggests the strong convergence of iterates in a stochastic subgradient type sequence to stationary points for this class of potentials. Subsequently, in Section~\ref{sec:sdiexist} we prove that there exists a strong solution to~\eqref{eq:langinc}, confirming the existence of a trajectory in the general case. In this sense, we extend the work of~\cite{leobacher2017strong,LeobacherSteinicke2022} studying diffusions with discontinuous drift to set-valued drift.
Next in Section~\ref{sec:fpandfreenergy} we prove the correspondence of a Fokker-Planck type equation to modeling the probability law associated with this stochastic process, and show that it asymptotically minimizes a free-energy functional corresponding to the loss function of interest, extending the seminal work of~\cite{jordanKinderlehrerOtto1998} which had proven the same result in the case of continuous $F(x)$. We present some numerical results that confirm the expected asymptotic behavior of~\eqref{eq:langinc} in Section~\ref{sec:num} and summarize our findings and their implications in Section~\ref{sec:conc}.



\subsection{Related Work}
Stochastic differential inclusions are the subject of the monograph~\cite{kisielewicz2013stochastic}, which presents the associated background of stochastic differential equations (SDEs) as well as set-valued analysis and differential inclusions, providing a notion of a weak and strong solution to equations of the form ~\eqref{eq:langinc}, and those with more general expressions, especially with respect to the noise term. In this work, and others studying stochastic differential inclusions in the literature, such as, e.g.,~\cite{kisielewicz2009stochastic,kisielewicz2020set}, it is assumed that the set-valued drift term $F(x)$ is Lipschitz continuous. In this paper we are interested in the more general case, where $F(x)$ may not be Lipschitz.

Langevin diffusions have had three distinct significant periods of development. To begin with, the original elegant correspondence of SDEs with semigroups associated with parabolic differential operators was explored in depth by~\cite{stroock2007multidimensional} (first edition published in the 1970s), following the seminal paper of~\cite{ito1953fundamental}. See also~\cite{kent1978}.

Later, they served as a canonical continuous Markovian process with the height of activity on ergodicity theory, the long run behavior of stochastic processes, associated with the famous monograph of~\cite{meyn2012markov}, first appearing in the 1990s. See, e.g.~\cite{meyn1993stability}.

Most recently, Langevin diffusion has been a model for studying the distributional dynamics of noisy training in contemporary machine learning models~\citep{hu2017diffusion}. At this point, we have the works closest to ours. In deep neural networks (DNNs), reconstructed linear units (ReLUs), which involve a component-wise maximum of zero and a linear expression, are standard functional components in predictors appearing in both regression and classification models. Mean-field analyses seek to explain the uncanny generalization abilities of DNNs by considering the distributional dynamics of idealized networks with infinitely many hidden layer neurons by considering both the stochastic equations corresponding to said dynamics, as well as Fokker-Planck-type PDEs modeling the flow of the distribution of the weights in the network. Analyses along these lines in the contemporary literature include~\cite{luo2021phase, shevchenko2021mean}. Although this line of work does use limiting arguments involving stochastic and distributional dynamics, they do not directly consider the Langevin differential inclusion and the potential PDE solution for the distribution, as we do here. 

\textcolor{black}{Next, we mention algorithmically focused papers on gradient-based sampling for potentials with non-smooth components. These fall into two categories: 1) methods that handle structured non-smoothness, such as a composite sum with a convex non-smooth term with a computable prox in~\cite{mou2022efficient,bernton2018langevin,durmus2019analysis,shen2020composite}, and 2) methods that incorporate smoothing techniques of the potential \citep{erdogdu2021convergence,chatterji2020langevin}. On the one hand, we analyze a very general class of nonsmooth potentials, considering the unmodified SDE and associated distribution law. On the other hand, the generality of the setting we consider is not amenable to the derivation of quantitative mixing times.}

We finally note that the Langevin diffusion as minimizing free energy even in the case of non-smooth potentials should not be surprising. In fact, in the seminal book of \cite{ambrosio2005gradient}, it is shown that a Clarke subgradient of a regular function still exhibits a locally minimizing flow structure for a functional potential on a probability space. Thus, at least locally, one may hope the necessary propertiesld.

\section{Tame functions, o-minimal structures and the exceptional set}\label{sec:exceptionalset}

Deep learning raises a number of important questions at the interface of optimization theory and stochastic analysis \citep{bottou2018optimization}. 
In particular, there are still major gaps in our understanding of the applications of plain stochastic gradient descent (SGD) in deep learning,
leaving aside the numerous related recent optimization algorithms for deep learning \citep{schmidt2021descending, davis2020stochastic}.
A particularly challenging aspect of deep learning is the composition of functions that define the objective landscape. These functions are typically recursively evaluated piecewise nonlinear maps. The nonlinearities are due to the sigmoidal function, exponentiation, and related operations, which are neither semialgebraic nor semianalytic, and the piecewise nature of the landscape is due to the common appearance of component-wise max.

We consider Euclidean space $\mathbb{R}^n$ with the canonical Euclidean scalar product $\left\langle \cdot,\cdot\right\rangle$
and a locally Lipschitz continuous function $f \colon \mathbb{R}^n \to \mathbb{R}$. For 
any $x \in \mathbb{R}^n$,
the Clarke subgradient of $f$ \citep{clarke1990optimization}:
\begin{equation}\label{eq:ClarkeSubgradient}
\partialc f(x) = \mathrm{conv} \left\{ v \in \mathbb{R}^n\,:\, \exists  y_k \underset{k \to \infty}{\longrightarrow} x \text{ with } y_k \in \mathbb{R},\, v_k = \nabla f(y_k) \underset{k \to \infty}{\longrightarrow} v \right\}.
\end{equation}

Following a long history of work \cite[e.g.]{macintyre1993finiteness}, we utilize a lesser known class of definable functions, which are known \cite{van1994elementary} to include restricted analytic fields with exponentiation. We refer to 
Macintyre, McKenna, and van den Dries \cite{macintyre1983elimination} and 
Knight, Pillay, and Steinhorn \cite{knight1986definable}
for the original definitions, and to  
van den Dries-Miller \cite{van1996geometric,van1998tame} and Coste \cite{coste1999introduction} for excellent book-length surveys. In particular, we use the following:

\begin{definition}[Structure, cf.  \cite{pillay1986definable}]
A {\em structure} on $(\mathbb{R},+,\cdot)$ is a collection of sets $\mathcal{O} = (\mathcal{O}_p)_{p \in \mathbb{N}}$, where each $\mathcal{O}_p$ is a family of subsets of $\mathbb{R}^p$ such that for each $p\in \mathbb{N}$:
\begin{enumerate}
\item  $\mathcal{O}_{p}$ contains the set $\{x\in \mathbb{R}^p: g(x)=0\}$, where $g$ is  a polynomial on $\mathbb{R}^p$, i.e., $g \in \mathbb{R}[X_1, X_2, \ldots, X_p]$, and the set is
often known as the family of real-algebraic subsets of $\mathbb{R}^p$;
\item if any $A$ belongs to $\mathcal{O}_p$, then $A\times \mathbb{R}$ and $\mathbb{R}\times A$ belong to $\mathcal{O}_{p+1}$;
\item if any $A$ belongs to $\mathcal{O}_{p+1}$, then $\pi(A) \in \mathcal{O}_{p}$, where  $\pi\colon\mathbb{R}^{p+1}\to\mathbb{R}^{p}$ is the coordinate or ``canonical'' projection onto $\mathbb{R}^p$, i.e., the projection to the first $p$ coordinates;
\item $\mathcal{O}_p$ is stable by complementation, finite union, finite intersection and contains $\mathbb{R}^p$, which defines a Boolean algebra of subsets of $\mathbb{R}^p$.
\end{enumerate}
\end{definition}

\begin{definition}[Definable functions, cf. \cite{van1996geometric}]
A {\em structure} on $(\mathbb{R},+,\cdot)$ is called {\em o-minimal} when the elements of $\mathcal{O}_{1}$ are exactly the finite unions of (possibly infinite) intervals and points.
Sets $A \subseteq \mathbb{R}^p$ belonging to an o-minimal structure $\mathcal{O}_{p}$, for some $p\in \mathbb{N}$, are called {\em definable in the o-minimal structure} $\mathcal{O}$.
One often shortens this to {\em definable} if the structure $\mathcal{O}$ is clear from the context.
A set-valued mapping is said to be definable in $\mathcal{O}$, whenever its graph is definable in $\mathcal{O}$.
\end{definition}

A subset of $\mathbb{R}^n$ is called {\em tame} if there exists an o-minimal structure such that the subset is definable in the o-minimal structure; \textcolor{black}{further, a function is called \emph{tame} if so is its graph}. Notice that this notion of tame geometry goes back to \emph{topologie mod{\' e}r{\' e}e} of  \cite{grothendieck1997around}.

Next, we consider two more regularity properties. First, we consider functions that have conservative set-valued fields of \cite{bolte2021conservative}, or equivalently, are path differentiable \citep{bolte2021conservative}.
This includes convex, concave, Clarke regular, and Whitney stratifiable functions. 

\begin{definition}[Conservative set-valued fields of \cite{bolte2021conservative}]
Let $D: \mathbb{R}^{n} \rightrightarrows \mathbb{R}^{n}$ be a set-valued map. $D$ is a {\em conservative (set-valued)  field} whenever it has closed graph, nonempty compact values and for any absolutely continuous loop $\gamma \colon [0,1] \to \mathbb{R}^n$, that is   $\gamma(0) = \gamma(1)$, we have
\begin{align*}
\int_0^1 \max_{v \in D(\gamma(t))} \left\langle \dot{\gamma}(t), v \right\rangle {\rm d} t = 0
\end{align*}
in the Lebesgue sense.
\label{def:conservativeMapForF}
\end{definition}

\begin{definition}[Potential function of  \cite{bolte2021conservative}]
Let $D$ be a {\em conservative (set-valued)  field}. 
For any $\gamma$ absolutely continuous with $\gamma(0) = 0$ and $\gamma(1)=x$,
any function $f$ defined as  \begin{eqnarray}
f(x)& = & f(0)+\int_0^1 \max_{v \in D(\gamma(t))} \left\langle \dot{\gamma}(t), v \right\rangle \dd t \label{pot1}\\
&= & f(0)+\int_0^1 \min_{v \in D(\gamma(t))} \left\langle \dot{\gamma}(t), v \right\rangle \dd t \label{pot2}\\
& = & f(0)+\int_0^1 \left \langle \dot{\gamma}(t), D(\gamma(t)) \right\rangle \dd t \label{pot3}\end{eqnarray}
is called a {\em potential function for~$D$}. We shall also say that {\em $D$ admits $f$ as a potential} or that {\em $D$ is a conservative field for~$f$}.
\rev{
Let us note that $f$ is well defined and unique up to a constant.
}
\end{definition}

The second notion of regularity, which we consider, is the notion of piecewise Lipschitzianity on $\mathbb{R}^{n}$ of \cite{leobacher2017strong}.
The associated exception(al) set \citep{LeobacherSteinicke2022} is the subset of the domain of that function where the function is not Lipschitz. We recall the definition here for convenience, and we state it for set-valued maps. 

\begin{definition}\label{def:piecewiseLipschitz}
A function $f:\mathbb{R}^{n}\to\mathbb{R}^{n}$ is \emph{piecewise Lipschitz continuous} if there exists a hypersurface $\Theta$, which we call the \emph{exceptional set} for $f$, such that $\Theta$ has finitely many connected
components $\theta_{i}$, $i=1,2,\ldots$, such that $f\big\vert_{\mathbb{R}^{n}\setminus\Theta}$ is intrinsic Lipschitz (cf. \cite[Definition 3.2]{leobacher2017strong}).

A set-valued function $F:\mathbb{R}^n\rightrightarrows\mathbb{R}^n$ is \emph{piecewise Lipschitz continuous} if there exists an exceptional set $\Theta$, defined analogously to single-valued functions, such that $F$ is Lipschitz on $\mathbb{R}^n\setminus\Theta$ with respect to Haudsdorff $H$-metric (cf. \cite{baierFarkhi2013}).
\end{definition}


Let us recall that, given a set-valued map with compact convex values $F:\mathbb{R}^n\rightrightarrows\mathbb{R}^n$, we say that $F$ is $H$-Lipschitz, or just Lipschitz, if and only if there exists a constant $l>0$ such that for any $x,y\in\mathbb{R}^n$ we have
\[
\rho_{H}(F(x),F(y))\leq l\|x-y\|,
\]
where, for arbitrary compact sets $A,B\subseteq\mathbb{R}^n$,
\[
\rho_H(A,B)\udef\max\{\max_{a\in A}\textrm{dist}(a,B),\max_{b\in B}\textrm{dist}(b,A)\},
\]
with $\textrm{dist}(a,B)\udef\min_{b\in B}\|a-b\|_2$.

\begin{definition}[\cite{van1996geometric,bolte2007clarke}] \label{def:stratification}
 A $C^r$ stratification of a closed (sub)manifold $M$ of $\mathbb{R}^n$ is a locally finite partition $(M_i)_{i\in I}$ of $M$ into $C^r$ submanifolds 
 (called \emph{strata})
 having the property that for $i\neq j$, 
 $\cl (M_i) \cap M_j\neq \emptyset$ implies that $M_j$ is entirely contained in $\cl(M_i) \setminus M_i$ (called the \emph{frontier} of $M_i$) and $\dim(M_j) < \dim(M_i)$. \\
 Moreover, a $C^r$ stratification $(M_i)_{i\in I}$ of $M$ has the Whitney-(a) property if, for each $x\in M_i\cap M_j$ (with $i\neq j$) and for each sequence $\{x_k\}\subseteq M_i$ we have
 \[
 \begin{cases}
\; \lim_{k\to\infty} x_k = x, &  \\
\; \lim_{k\to\infty} T_{x_k}M_i = \mathcal{T}, & 
 \end{cases}
\implies T_x M_j\subseteq\mathcal{T},
\]
where $T_x M_j$ (respectively, $T_{x_k}M_i$) denotes the tangent space of the manifold $M_j$ at $x$ (respectively, of $M_i$ at $x_k$) and the convergence in the second limit is in the sense of the standard topology of the Grassmannian bundle of the $\dim M_i-$planes in $T_xM$, $x\in M_i$ (see \cite{MatherNotes}). \\
If, moreover, a Whitney stratification $(M_i)_{i\in I}$ satisfies for all $i\in I$ and $x\in M_i$ the transversality condition
\begin{equation}\label{eq:H}
e_{n+1}\notin T_x M_i,
\end{equation}
where $e_{n+1} = \begin{bmatrix} 0 & \cdots & 0 & 1\end{bmatrix}^\top\in\mathbb{R}^{n+1}$, then it is called a \emph{nonvertical} Whitney stratification.
 
 A function $f:\mathbb{R}^n\to\mathbb{R}$ is said to be \emph{stratifiable}, in any of its connotations, if the graph of $f$, denoted by $\textrm{Graph}(f)$, admits a corresponding $C^r$ stratification.
\end{definition}

\begin{remark}\label{rem:stratification}
Let us note that, from Definition \ref{def:stratification}, if $(M_i)_{i\in I}$ is a stratification of a stratifiable submanifold $M$, then it follows that $M=\bigcup_{i\in I}M_i$. Therefore, by the Hausdorff maximality principle, there must exist $\overline{i}\in I$ such that $\dim M_{\overline{i}}=\max_{i\in I}\dim M_i$ and, thus, $M=\cl{M_{\overline{i}}}$ since all the other subspaces $M_i,\,i\neq\overline{i}$ have zero Lebesgue measure. The same argument holds, a fortiori, for every finite subset $J\subseteq I$.
\end{remark}

In this paper, we establish the existence guarantees of strong solutions to the equation~\eqref{eq:langinc} and study the PDE describing the flow of the probability mass of $X(t)$. To set up the subsequent exposition, we must first establish the groundwork of linking Whitney stratifiable potential functions, which describe the loss landscape of neural networks and other statistical models on the one hand, to set valued piecewise Lipschitz continuous maps, which describe the distributional flow of stochastic subgradient descent training, modeled by~\eqref{eq:langinc}, on the other.

However, in order to do so, we must make an additional assumption that limits the local oscillation and derivative growth of the potential; specifically, we assume that $f$, the potential of $F$, has a bounded variation. It can be seen that the standard activation and loss functions that appear in neural network training satisfy this condition. 
\begin{theorem}\label{thm:FpwLip}
Let $F \colon \mathbb{R}^n \rightrightarrows \mathbb{R}^n$ be a definable conservative field that admits a tame potential $f\colon \mathbb{R}^n \to \mathbb{R}$ with bounded variation. Let $F$ be Lipschitz on $\mathbb{R}^n\setminus B_\delta(0)$ with constant $L$, for some $L,\delta>0$. Then  $F$ is piecewise Lipschitz continuous. 
\end{theorem}
\begin{proof}
As $f$ is tame, by definition, it is also definable. Since $F$ is conservative, from \cite{bolte2021conservative}, $f$ is locally Lipschitz and, since it is tame, \cite[Theorem 1]{BolteEtAl2009} implies that $f$ is semismooth. Therefore, by \cite[Corollary 9]{bolte2007clarke}, letting $r\geq2$ be an arbitrarily fixed integer, $f$ admits a nonvertical $C^r$ Whitney stratification $(M_i)_{i\in I}$. With abuse of notation, let $(M_i)_{i\in I}$ denote the stratification relative to the domain of $f$. Therefore, due to local finiteness, there must exist a maximal finite subset of indices $J\subseteq I$ such that 
\[
B_\delta(0)\cap M_j\neq\emptyset,\quad j\in J.
\]
Since $f$ is semismooth we deduce that its directional derivative $f'(\cdot;v)$ is $C^{r-1}$ for all $v\in\mathbb{R}^n$ and, since $f\in\textrm{BV}(\mathbb{R}^n)$, $f'(\cdot;v)$ is bounded on $B_\delta(0)$; moreover, for all $j\in J$, letting $x\in B_\delta(0)\cap M_{j}$ and $v\in \mathcal{T}_x(B_\delta(0)\cap M_{j})$ it holds that $f'(x,v)$ is Lipschitz continuous with respect to $x$, restricted to $B_\delta(0)\cap M_{j}$. Hence, 
the Riemannian gradient $\nabla_R f(x)$ is Lipschitz on $B_\delta(0)\cap M_{j}$, $x\in M_j$, $j\in J$, since $\nabla_R f(x)$ is the restriction of the directional derivative to tangent directions onto $M_j$ \citep{bolte2007clarke}. Let us note that the Clarke subgradient of $f$ at $x$ is such that
\[
\textrm{Proj}_{T_x M_j}\partialc 
f(x)\subseteq\{\nabla_R f(x)\},
\]
where $M_j$ is the stratum such that $x\in M_j$, $j\in J$ (see  \cite[Proposition 4]{bolte2007clarke}).\\
Now, on the account of Remark \ref{rem:stratification}, for some $\overline{j}_x\in J$,
we have $\dim M_{\overline{j}_x}=\dim B_{\delta}(0)=n$. By compactness, there exists a finite covering of $B_{\delta}(0)$ of maximal dimension, denoted by $(M_{\overline{j}})_{\overline{j}\in\overline{J}}$ for some $\overline{J}\subseteq J$, such that $B_{\delta}(0)\subseteq\overline{M}$, where $\overline{M}\udef\bigcup_{\overline{j}\in\overline{J}}M_{\overline{j}}$. Therefore $\nabla_R f(x)=\nabla f(x)$ on $B_{\delta}(0)$.

Then, as a consequence of \cite[Theorem 1]{bolte2021conservative}, there exists a zero-measure set $S\subseteq\mathbb{R}^n$ such that $F(x)=\{\nabla f(x)\}$ for all $x\in B_{\delta}(0)\setminus S$. Now, for $x\in B_{\delta}(0)\setminus S$, from \cite[Corollary 1]{bolte2021conservative}, we have that 
\[
\partialc f(x)\subseteq\textrm{conv}F(x)=\textrm{conv}\{\nabla f(x)\}=\{\nabla f(x)\},
\]
so that
\[
\textrm{Proj}_{T_x M_j}\partialc 
f(x)\subseteq\textrm{Proj}_{T_x M_j}\{\nabla f(x)\}=\{\nabla f(x)\}.
\]
It then follows that $\nabla f(x)$, and so $F$, 
is Lipschitz on $B_\delta(0)\setminus S$. Since $B_\delta(0)$ is compact, there exists a finite family $\{\theta_k\}_{k\in K}$ such that $B_\delta(0)\cap S\subseteq\cup_{k\in K}\theta_k$. Thus $F$ is Lipschitz on $B_\delta(0)\setminus\cup_{k\in K}\theta_k$, and the claim is proved.
\end{proof}

\begin{example}
In case $F \colon \mathbb{R}^n \rightrightarrows \mathbb{R}^n$ is a definable conservative field such that none of its tame potentials $f\colon \mathbb{R}^n \to \mathbb{R}$ has bounded variation, then Theorem \ref{thm:FpwLip} does not hold. In fact, let us consider
\rev{
\[
F(x)\udef
\begin{cases}
    0, & x\in(-\infty,0), \\
    [0,1], & x=0, \\
    \frac{1}{x}, & x\in(0,1].
\end{cases}
\]
}
A simple computation provides that $F$ is a definable conservative field. 
\rev{
Moreover, its unique potential, up to constants, is not of bounded variation, since
\[
\int_0^{+\infty}F(x)\,\de x=-\infty. 
\]
In this case, Theorem \ref{thm:FpwLip} does not hold, as $F$ is not Lipschitz, and a fortiori not piecewise Lipschitz.
}
\end{example}


\section{Existence and Uniqueness of Solution to the SDI}\label{sec:sdiexist}

In this section we will prove that \eqref{eq:langinc} admits a strong solution. More precisely, it will be proven that there exists a suitable selection of $F(X_{t})$ such that the corresponding SDE has a strong solution: this will in turn imply that the original stochastic differential inclusion has a solution as well. Our result will rely on an existence and uniqueness pertaining to SDEs with discontinuous drift in \cite{leobacher2017strong}. 

\subsection{Piecewise Lipschitz selections of upper semicontinuous set-valued maps}\label{sec:results}

In our setting, assuming that $F$ is the Clarke subdifferential of some continuous tame function guarantees that $F$ is an upper semi-continuous set-valued map. We further assume that $F$ is bounded with compact convex values. Our aim is to prove that, under these assumptions, $F$ has a piecewise Lipschitz selection.

We are going to need the following results:
\begin{theorem}[Theorem 9.4.3 in \cite{aubinFrankowska}]\label{thm:steinerSelection} Consider a Lipschitz set-valued map $F$ from a metric space to nonempty closed convex subsets of $\mathbb{R}^{n}$. Then $F$ has a Lipschitz selection, called \emph{Steiner selection}.
\end{theorem}

\begin{theorem}[Kirszbraun's Theorem, cf. \cite{federer}]\label{thm:kirszbraun}
If $S\subseteq\mathbb{R}^n$ and $f:S\to\mathbb{R}^n$ is Lipschitz, then $f$ has a Lipschitz extension $g:\mathbb{R}^n\to\mathbb{R}^n$ with the same Lipschitz constant.
\end{theorem}

In the next Theorem, which is the main result of this section, we prove that, under some mild assumptions, the set-valued map $F$ in \eqref{eq:langinc} has a piecewise Lipschitz selection for any suitable compact covering of $\mathbb{R}^n$. 

\rev{
\begin{theorem}\label{thm:locallyLipschitzSelection}
Let $F:\mathbb{R}^{n}\rightrightarrows\mathbb{R}^{n}$ be an upper semi-continuous set-valued map with closed convex values, and piecewise Lipschitz continuous, with exceptional set $\Theta$. Then $F$ has a piecewise Lipschitz selection with exceptional set $\Theta$, arbitrarily smooth on the interior of each connected component. 
\end{theorem}
\begin{proof}
Let $\{R_i\}_{i=1}^{r}$ be the finite family of closed subsets of $\mathbb{R}^n$ such that 
\[
\bigcup_{i=1}^{r}R_i=\mathbb{R}^n\setminus\Theta.
\]
For each $i=1,\ldots,r$, Theorem \ref{thm:steinerSelection} on $F\big\vert_{R_{i}}$ implies that there exists a finite sequence of equi-Lipschitz, and thus continuous, selection functions $\{f_{i}\}$ where each $f_{i}$ is defined on $R_i$, that is $f_{i}:R_i\to\mathbb{R}^{n}$.

We can now extend each $f_{i}$ on the whole $\mathbb{R}^n$, to some function, still denoted by $f_{i}$, which is Lipschitz with the same constant as the original function's on the account of Kirszbraun's Theorem \ref{thm:kirszbraun}.
\\
Let now $\varepsilon>0$ be given and let $i=1,\ldots,r$ be fixed. Let us consider a partition of unity $\{\varphi^\varepsilon\}$ as in \cite[Lemma 1.2]{Shubin1990}.
Now, following a classical construction (e.g., see \cite{azagraEtAl2007}), we let
\[
g_i^\varepsilon(x)\udef\int_{\mathbb{R}^n}f_i(x)\varphi^\varepsilon(y-x)\,\de y.
\]
It then follows that $g_{i}^{\varepsilon}\in\mathscr{C}^{\infty}(\mathbb{R}^{n})$ is a Lipschitz function, with the same Lipschitz constant as $f_{i}$'s, and is such that $g_i^\varepsilon\in F(x)$ by straightforward computations.
Let us stress that the Lipschitz constant of each $g_{i}^{\varepsilon}$ is independent of $\varepsilon$, so that
$\{g_{i}^{\varepsilon}\}_{\varepsilon>0}$ is equi-Lipschitz on $R_i$.
Let now $x\in\mathbb{R}^n$ and $i_x\in\N$ be unique index such that $x\in R_{i_x}$. 
We then define
\[
f(x)\udef g_{i_{x}}^\varepsilon(x),\quad x\in\mathbb{R}^{n}.
\]
It then follows that $f:\mathbb{R}^{n}\to\mathbb{R}^{n}$ is piecewise Lipschitz on $\mathbb{R}^{n}$ according to Definition \ref{def:piecewiseLipschitz} and its exceptional set is $\Theta$. Obviously $f(x)\in F(x)$, and this proves the claim.
\end{proof}
}

\begin{remark}\label{rem:compactCovering}
Theorem \ref{thm:locallyLipschitzSelection} still holds if we replace $\mathbb{R}^n$ with any $A\subseteq\mathbb{R}^n$: in fact, we can always extend $F$ to an upper semi-continuous set-valued map defined on the whole $\mathbb{R}^n$ by virtue of \cite[Theorem 2.6]{smirnov}.
\end{remark}

Finally, we have the following.
\rev{
\begin{corollary}\label{cor:SDI_existenceUniqueness}
Let $F:\mathbb{R}^n\rightrightarrows\mathbb{R}^n$ be an upper semi-continuous set-valued map with closed convex values, and piecewise Lipschitz continuous with a $\mathscr{C}^3$ exceptional set $\Theta$ of positive reach. Then the SDI \eqref{eq:langinc} admits a strong solution. In particular, for every piecewise Lipschitz selection of $F$, there exists a unique strong solution to the SDI \eqref{eq:langinc}.
\end{corollary}
\begin{proof}
From Theorem \ref{thm:locallyLipschitzSelection} there exists a drift $\mu:\mathbb{R}^{n}\to\mathbb{R}^{n}$, piecewise Lipschitz selection of the set-valued map $F$, which satisfies Assumptions $3.4-3.6$ from \cite{leobacher2017strong}. Moreover, since diffusion is constant is \eqref{eq:langinc}, Assumption $3.3$ therein is satisfied as well, while Assumptions $3.1-3.2$ comes from hypotheses. Therefore, on the account of \cite[Theorem 3.21]{leobacher2017strong} , we obtain that the SDE 
\begin{equation}\label{eq:langSDE}
\de X_t=-\mu(X_t) \de t+ \sqrt{2\sigma}\,\de B_t
\end{equation}
has a unique global strong solution, which in turn represents a strong solution to the stochastic differential inclusion \eqref{eq:langinc}, and this proves the claim.
\end{proof}
}
\rev{
\begin{remark}
 It is worth noticing that the smoothness assumption on the exceptional set $\Theta$ in Corollary \ref{cor:SDI_existenceUniqueness} is unavoidable for resorting to \cite{leobacher2017strong} and, at the best of our knowledge, nothing more is known about existence of strong solutions to \eqref{eq:langSDE} under the assumptions required in this work. However, the topic seems to be under investigation, and more details could be found in \cite{LeobacherSteinicke2022}.   
\end{remark}
}
\rev{
Now, if $F$ is the Clarke subdgradient of a continuous tame function $f$, i.e. $F=\partialc f$ as in \eqref{eq:ClarkeSubgradient}, a stronger uniqueness result can be obtained. 
\begin{theorem}\label{thm:SDI_existenceUniqueness}
Under the assumptions of Corollary \ref{cor:SDI_existenceUniqueness}, the SDI \eqref{eq:langinc} admits a strong solution that is unique almost everywhere.
\end{theorem}
\begin{proof}
Let $\mu_1,\mu_2:\mathbb{R}^d\to\mathbb{R}^d$ be two piecewise Lipschitz selections from $F(x),x\in\mathbb{R}^d$ with the same exceptional set $\Theta$. Since $f$ is a continuous tame function, then it is locally Lipschitz, and henceforth differentiable almost everywhere. Thus, there exists a zero Lebesgue-measure set $N\subseteq\mathbb{R}^n$ such that
\[
F(x)=\{\nabla f(x)\}\quad\forall x\in\mathbb{R}^n\setminus\left(N\cup\Theta\right).
\]
This implies that, for almost all $x\in\mathbb{R}^d$, $\mu_1(x)=\mu_2(x)$. Now, let $X_t^{(1)},X_t^{(2)}$ be the solutions with the same Brownian motion realization to
\[
\de X_t^{(i)}=\mu_i(X_t^{(i)})\de t+\sqrt{2\sigma}\de B_t,\quad X_0^{(i)}=X_0,\quad i=1,2.
\]
Now, let $\varepsilon>0$ be arbitrary and, noting that $X_0^{(1)}=X_0^{(2)}$, let $\overline{t}_1$ be the first time when $X_t^{(1)}, X_t^{(2)}$ hit the first connected component $\theta_1$ of $\Theta$. Then, for $t\in[0,\overline{t}_1]$, by Markov inequality we have
\begin{align*}
\mathbb{P}\left(\|X_t^{(1)}-X_t^{(2)}\|\geq\varepsilon\right) &\leq \frac{\mathbb{E}\left[\|X_t^{(1)}-X_t^{(2)}\|\right]}{\varepsilon} \\
&= \frac{\mathbb{E}\left[\int_0^t\|\mu_1(X_\tau^{(1)})-\mu_2(X_\tau^{(2)})\|\de\tau\right]}{\varepsilon} \\
&= \frac{\mathbb{E}\left[\int_0^t\|\mu_1(X_\tau^{(1)})-\mu_1(X_\tau^{(2)})+\mu_1(X_\tau^{(2)})-\mu_2(X_\tau^{(2)})\|\de\tau\right]}{\varepsilon}  \\
&\leq \frac{\mathbb{E}\left[\int_0^t\|\mu_1(X_\tau^{(1)})-\mu_1(X_\tau^{(2)})\|\de\tau\right]}{\varepsilon}+\frac{\mathbb{E}\left[\int_0^t\|\mu_1(X_\tau^{(2)})-\mu_2(X_\tau^{(2)})\|\de\tau\right]}{\varepsilon} \\
&= 0,
\end{align*}
being 
\[
\int_0^t\|\mu_1(X_\tau^{(1)})-\mu_1(X_\tau^{(2)})\|\de\tau=0
\]
since $X_\tau^{(1)}=X_\tau^{(2)}$ for all $\tau\in[0,\bar{t}_1)$, and
\[
\int_0^t\|\mu_1(X_\tau^{(2)})-\mu_2(X_\tau^{(2)})\|\de\tau=0
\]
since $\mu_1(x)=\mu_2(x)$ for almost every $x\in\mathbb{R}^d$ and they are bounded in $\mathbb{R}^d$, ruling out the possibility for any of the $\mu_i$'s, $i=1,2$, to behave as a Dirac delta in a neighborhood of $\bar{t}_1$. \\
Thus, since $\varepsilon>0$ has been arbitrarily chosen, it follows that $X_t^{(1)}=X_t^{(2)}$ for a.e. $t\in[0,\overline{t}_1]$. Now, proceeding by finite induction on the number of connected components of the exceptional set $\Theta$, with inductive hypothesis $X_{\overline{t}_i}^{(1)}=X_{\overline{t}_i}^{(2)}$, and applying analogous arguments as above, the claim is proved.
\end{proof}
}

\section{Fokker-Planck Equation and Free Energy Minimization}\label{sec:fpandfreenergy}

\subsection{Fokker-Planck Equation}

The Fokker-Planck (FP) equation describes the evolution of the probability density associated with the random process modeled by the diffusion flow. Classical results deriving the Fokker-Planck equation from SDEs can be reviewed in~\cite{eklund1971boundary,ito1953fundamental,kent1978,stroock2007multidimensional}. In particular, for a drift diffusion of the form $\de X_t=-\nabla f(X_t)\,\de t+\sqrt{2\sigma}\de B_t$ it holds that, from any initial distribution $\rho(0)=\rho_0$, the density $\rho(x,t)$ of $X_t$ is given by 
\begin{equation}\label{eq:fp}
    \frac{\partial \rho}{\partial t} = -\nabla\cdot\left(\nabla f(x)\rho\right)+\frac{1}{2}\Delta\left(\sigma \rho\right)
\end{equation} 
and has a limiting stationary distribution defined by the Gibbs form $\exp\left\{-f(x)/\sigma\right\}/Z$. In~\cite{jordanKinderlehrerOtto1998},  it was shown that the Fokker-Planck evolution corresponds to the gradient flow of a variational problem of minimizing a free-energy functional composed of the potential $f(x)$ and an entropy regularization. 

These classical results in order to even concern well-defined objects, require the smoothness of $f$. The FP equation can be derived from applying integration by parts after It{\^o}'s Lemma on the diffusion process. Following~\cite[4.3.5]{gardiner2009stochastic}, for arbitrary $g$, we have,
\begin{align*}
\frac{\langle g(x(t))\rangle}{\de t} &= \frac{\de}{\de t}\langle g(X(t))\rangle \\
&= \langle \nabla f(x) \partial_x f+\frac 12 \sigma \partial^2_{xx} g\rangle \\ 
&= \int\,\de x\left[\nabla f(x)\partial_x g+\frac 12 \sigma \partial^2_{xx} g\right]\rho(x,t\vert x_0,t_0) \\
&=
\int\,\de x g(x)\partial \rho(x,t\vert x_0,t_0)
\end{align*}
Replacing the expression $\nabla f(x)$ with an arbitrary coefficient function of form $a(x)$, we can see that in the case that $a(x)$ is a selection of the Clarke subdifferential it may not be a continuous function of $x$. In this case, even the weak sense derivative of $a(x)$ does not exist and integration by parts cannot be applied. 

\subsection{Generic Solution Existence Results}

Nevertheless, a solution $\rho(x,t)$ can be shown to exist for the system as stated in weak form. To this end, we follow \cite{BogachevKrylovRockner2001}. Let $\Omega_T\subseteq\mathbb{R}^{n}\times [0,T]$ be open, and
\[
L_{a,\sigma}\varphi=\sum_{i=1}^{n}a_{i}\parder{\varphi}{x_{i}}+\sum_{i,j=1}^{n}\sigma_{ij}\frac{\partial\varphi}{\partial x_{i}\partial x_{j}},\quad\varphi\in\mathscr{C}_{0}^{\infty}(\Omega_T).
\]
Moreover, let $n'$ be such that $\frac1n+\frac{1}{n'}=1$. The following holds:
\begin{theorem}[Corollary 3.2, \cite{BogachevKrylovRockner2001}]\label{th:generalexistfp}
Let $\mu$ be a locally finite Borel measure on $\Omega$ such that $a_i,\sigma_{ij}\in L_{loc}^{1}(\Omega_T,\mu)$ and, 
\[
\int_{\Omega_T}\left[\frac{\partial \phi}{\partial t}+\sum_{i=1}^{n} a_i \frac{\partial \phi}{\partial x} +\sum_{i,j=1}^{n}\sigma_{ij}\frac{\partial\varphi}{\partial x_{i}\partial x_{j}}\,\right]\de\mu=0
\]
for all nonnegative $\varphi\in\mathscr{C}_{0}^{\infty}(\Omega_T)$. Furthermore let $\sigma_{ij}$ be uniformly bounded, nondegenerate and H{\"{o}}lder continuous, then $\mu=\rho(x,t)\,\de x\,\de t$ with $\rho\in L_{loc}^{p}(\Omega_T)$ for every $p\in[1,(n+2)')$.
\end{theorem}
\begin{remark}
As observed in \cite{BogachevKrylovRockner2001}, one cannot expect that the density of $\mu$ is continuous even for infinitely differentiable $\sigma_{ij}$ under these conditions. However, we note that in~\cite{portenko1990generalized}
a continuous solution is shown under the assumption that $a_i\in L^{p}(\Omega_T,\mu)$, i.e. it is globally integrable. Again, however, this is very restrictive in the case of studying the evolution of diffusion operators on tame nonsmooth potentials of interest. 
\end{remark}
Existence of an invariant measure $\mu$ for the probability flow, i.e., a solution for the purely elliptic part, is guaranteed by \cite[Theorem 1.6]{BogachevRockner2001} and \cite[Theorem 1.2]{AlbeverioBogachevRockner1998}.

The strong regularity conditions, and the bounded open set $\Omega_T$, however, clearly limit both the applicability and informativeness of these results for our SDI.

\subsection{Fokker-Planck With Boundary Conditions}
Note, however, that we have an expectation as to what the stationary distribution for~\eqref{eq:langinc}, namely, of Gibbs form proportional to $\exp\{-f(x)/\sigma\}$. This measure is even absolutely continuous, and thus has higher regularity than Theorem~\ref{th:generalexistfp} suggests. 

To this end, consider~\cite[Chapter 5.1.1]{gardiner2009stochastic} which considers boundary conditions at a discontinuity for the FP equation. We can consider that the space is partitioned into connected components $R_i$, and in each region, the continuous SDE and associated FP~\eqref{eq:fp} holds. However, for boundaries between regions $S_{ij}\subseteq \Theta$ we have,
\begin{equation}\label{eq:boundaryfp}
n\cdot J_i(x,t)\big\vert_{S_{ij}^i} = n\cdot J_j(x,t)\big\vert_{S_{ij}^j},\,\,
\rho(x,t)\big\vert_{S_{ij}^i} = \rho(x,t)\big\vert_{S_{ij}^j}
\end{equation}
where the probability current $J$ is defined as,
\[
J_i(x,t) = F_i(x)\rho(x,t)+\frac 12 \sigma \frac{\partial\rho(x,t)}{\partial x_i}
\]
as defined on region $R_i$

Note that in one dimension this is simple: we have, at a point of non-differentiability $x_c$,
\begin{align*}
\rho(x,t)\big\vert_{x_c^+} &= \rho(x,t)\big\vert_{x_c^-},\\
\left.\left(F(x)\rho(x,t)+\frac 12 \sigma \frac{\partial\rho(x,t)}{\partial x_i}\right)\right\vert_{x_c^+} &= -\left.\left(F(x)\rho(x,t)+\frac 12 \sigma \frac{\partial\rho(x,t)}{\partial x_i}\right)\right\vert_{x_c^-}
\end{align*}
where we abuse notation to indicate the partial directional derivative of $\rho$ from either side of $x_c$. 

We can consider writing a stationary solution $\pi(x)$ to this system as we have a suspected ansatz, similarly as done in~\cite[Section 5.3]{gardiner2009stochastic}. Generically stationary implies $\nabla \cdot J(x)=0$, or that $J$ is divergence-free with respect to $x$. We have,
\begin{equation}\label{eq:solutionfp}
\sum_{i=1}^n \frac{\partial}{\partial x_i}  \left[F_i(x)\rho(x,t)+\frac 12 \sigma \frac{\partial\rho(x,t)}{\partial x_i}\right] = 0
\end{equation}

Indeed, let
\[
\pi(x) = \exp\left\{-f(x)/\sigma\right\}/Z.
\]
On the domain $\mathbb{R}^n\setminus\Theta$ where $\nabla f(x)$ is well-defined, this can
be seen immediately to solve~\eqref{eq:solutionfp}. Since $\Theta$ is of measure zero with respect to the ambient space, we can define,
\[
Z = \int\limits_{\mathbb{R}^n} \pi(x) \,\de x = \int\limits_{\cup_i R_i} \pi(x) \,\de x
\]

Of course, a constructive ansatz for the stationary solution, the elliptic form, neither shows its uniqueness, nor the existence, uniqueness and regularity of the parabolic evolution of the probability mass flow $\rho(x,t)$. To the best of our knowledge, no such result exists for the network of parabolic first-order systems under consideration, and at the same time, we shall see that the continuity of $\rho(x,t)$ is important in the next section for the variational conception of the FP equation as a minimizing flow for the free energy.

To this end, we can consider two lines of work as a foundation to formulate the requisite results. In particular, \cite{nittka2011regularity} shows regularity conditions of solutions to second-order parabolic equations defined on a Lipschitz domain. As seen in Section~\ref{sec:exceptionalset}, for the problems of interest, the exceptional set can be parameterized in a smooth way, which implies that for compact sets, the boundary is Lipschitz. We must, however, take care to translate the results appropriately to the potentially unbounded domains a potential $R_i$ could correspond to. 

The closest to our line of work is considering graph-structured networks of PDEs, such as modeling Kirchhoff's laws. A prominent and representative work along these lines is~\cite{von1988classical}. In this setting, there is a network of one-dimensional paths embedded in some ambient space in $\mathbb{R}^n$, with the paths connected at a series of vertices, with boundary conditions connecting a collection of linear parabolic PDEs governing the flow of a quantity across the network. Thus, the spirit of a network structure of PDEs with connecting boundary connections is analogous to our problem, however, with the caveat that they are one-dimensional embedded domains, even if embedded in a larger space. Existence and smoothness regularity conditions are shown.

We consider an approach using domain decomposition methods for the solutions of PDEs based on optimal control theory~\citep{gunzburger1999optimization}. 
See also, e.g.,~\cite{dolean2015introduction}. We consider reformulating the boundary conditions as a control to establish a variational formulation. 

First let $\delta>0$ be a regularization parameter and consider the optimization problem
\begin{equation}\label{eq:optcontroldelta}
\begin{split}
\min\limits_{\{\rho_i\in H^1(R_i)\}\{g_{ij}\in L^2(S_{ij})\}} &
\mathcal{J}_\delta(\{\rho_i\},\{g_{ij}\}) 
\\
\text{subject to } & \frac{\partial \rho_i}{\partial t} + F(x)\cdot \nabla \rho_i+\sigma \nabla\cdot\nabla \rho_i = 0,\,\text{ on }R_i,\,\forall i,\\
& n\cdot J_i = g_{ij},\,\text{ on }S_{ij},\\
& n\cdot J_j = -g_{ij},\,\text{ on }S_{ij},
\end{split}
\end{equation}
where
\[
\mathcal{J}_\delta(\{\rho_i\},\{g_{ij}\}) := \sum\limits_{ij}\left[\int_{S_{ij}}(\rho_i-\rho_j)^2\de S_{ij} +\frac{\delta}{2}\int_{S_{ij}} g_{ij}^2 \,\de S_{ij}\right].
\]
Let us also note the weak form of the PDE constraints,
\begin{equation}\label{eq:weakpdeoptcont}
\begin{split}
\int_{R_i} \left[\partial_t \rho_i v+\sigma\nabla \rho_i\cdot \nabla v +F\cdot \nabla \rho v\right] \,\de x  = \\
\sum\limits_j\int_{S_{ij}} g_{ij} \de S_{ij} -\sum\limits_j\int_{S_{ji}} g_{ji} \de S_{ji} ,\,
\forall v\in H^1(R_i).
\end{split}
\end{equation}

We have the following, akin to~\cite[Theorem 2.1]{gunzburger1999optimization}. Let
\[
\mathcal{U} = \left\{\{\rho_i\in H^1(R_i)\}\{g_{ij}\in L^2(S_{ij})\} \text{ satisfying~\eqref{eq:weakpdeoptcont}},\, \mathcal{J}_{\delta}(\{\rho_i\},\{g_{ij}\})<\infty\right\}
\]
\begin{theorem}\label{th:existoptcondd}
There exists a unique solution $(\{\rho_i\in H^1(R_i)\}\{g_{ij}\in L^2(S_{ij})\})$ to~\eqref{eq:optcontroldelta} in $\mathcal{U}$.
\end{theorem}
\begin{proof}
Let $\{\rho_i^{(n)},g_{ij}^{(n)}\}$ be a minimizing sequence in $\mathcal{U}$, i.e.,
\[
\lim\limits_{n\to\infty} \mathcal{J}_{\delta}\left(\{\rho_i^{(n)},g_{ij}^{(n)}\}\right)=\inf\limits_{\{\rho_i,g_{ij}\}\in\mathcal{U}} \mathcal{J}_{\delta}\left(\{\rho_i,g_{ij}\}\right)
\]
By the definition of $\mathcal{U}$ we have that $g^{(n)}_{ij}$ are uniformly bounded in $L^2(S_{ij})$. 

Now, we argue that by~\cite[Theorem IV.5.3]{ladyzhenskaya1967linejnye} the PDEs given by~\eqref{eq:weakpdeoptcont} have unique solutions $\rho_i$ continuous with respect to the inputs, i.e.,
\begin{equation}\label{eq:energyest}
\|\rho_i\|_{H^{l+2,l/2+1}}\le C \sum_j\left(\|g_{ij}\|_{H^{l+1,(l+1)/2}(S_{ij})}+\|g_{ji}\|_{H^{l+1,(l+1)/2}(S_{ji})}\right)
\end{equation}
for some non-integral $l$, when the norms on the right-hand side are well defined.

When we consider the conditions for the application of this Theorem, we see that the only unsatisfied assumption is the integrability of the coefficients with respect to an appropriate dual Sobolev space. However, we can see from the proof of the result, in particular~\cite[Equation (IV.7.1)]{ladyzhenskaya1967linejnye}, that this is used only to show that the operators,
\[
\left[\partial_t +\sigma \nabla^2+F\cdot \nabla\right] \text{ and } \left[\sigma \nabla +F\cdot \right]
\]
are bounded. However, with $\sigma$ constant and $F\in L^\infty(R_i)$, this also clearly holds and these are operators from $H^{l+2,l/2+1}$ to $H^{l,l/2+1}$ on $R_i$
and $H^{l+1,l/2+1}$ on $S_{ij}$ respectively.

For any $l$, we can apply the Poincar{\'e} inequality on the left and Sobolev embedding on the right of~\eqref{eq:energyest} to obtain,
\begin{equation}\label{eq:energyesttwo}
\|\rho_i\|_{H^{1}(R_i)}\le C \sum_j\left(\|g_{ij}\|_{L^2(S_{ij})}+\|g_{ji}\|_{L^2(S_{ji})}\right)
\end{equation}

Thus, by the uniform boundedness of $g_{ij}^{(n)}$ from the definition of $\mathcal{U}$, we get the boundedness of $\rho_i^{(n)}$ and the existence of a convergent subsequence $\{\{\rho_i^{(n_k)}\},\{g_{ij}^{(n_k)}\}\}$ convergent to $(\{\hat\rho_i\},\hat g_{ij}\}$ with every $\hat\rho_i$ in $H^1(R_i)$ and $\hat g_{ij}$ in $L^2(S_{ij})$. Passing to the limit we see that they also satisfy~\eqref{eq:weakpdeoptcont} and by the lower semicontinuity of $\mathcal{J}_{\delta}$ we have that
\[
\inf\limits_{\{\rho_i\},\{g_{ij}\}} \mathcal{J}_{\delta}(\rho_i,g_{ij}) = \lim\inf\limits_k \mathcal{J}_{\delta} (\rho_i^{(n_k)},g_{ij}^{(n_k)})\ge \mathcal{J}_{\delta}(\{\hat \rho_i,g_{ij}\})
\]
and $(\{\hat\rho_i\},\hat g_{ij}\}$ is optimal. Since $\mathcal{J}_{\delta}$ is convex and $\mathcal{U}$ is linear, it is uniquely optimal.
\end{proof}

Now consider the following weak system of PDEs, now for a unique $\rho$,
\begin{equation}\label{eq:weakdd}
\begin{split}
\int_{R_i} \left[\partial_t \rho v+\sigma\nabla \rho\cdot \nabla v +F\cdot \nabla \rho v\right] \,\de x =\sum_j \int_{S_{ij}} (n\cdot (F \rho+\frac 12\sigma \nabla \rho))v\,\de S_{ij}
\\ \qquad\qquad\qquad -\sum_j \int_{S_{ji}} (n\cdot (F \rho+\frac 12\sigma \nabla \rho))v \,\de S_{ji},\,
\forall v\in H^1(R_i)
\end{split}
\end{equation}
Then, we prove the following.
\begin{theorem}\label{th:limitoptcondd}
For each $\delta>0$, denote $(\{\rho^\delta_i\}\{g^\delta_{ij}\})$
the solutions to~\eqref{eq:optcontroldelta} as given by Theorem~\ref{th:existoptcondd}. We have that for any convergent subsequence as $\delta\to 0$, it holds that there exists $\rho$ such that with $\rho_i=\rho\vert_{R_i}$ and $g_{ij} = n\cdot (F\rho+\frac 12\sigma \nabla \rho)\vert_{S_{ij}}$,
\[
\sum_i \|\rho^\delta_i-\rho_i\|_{H^1(R_i)}+\sum_{ij} \|g^\delta_{ij}-g_{ij}\|_{L^2(S_{ij})}\to 0
\]
and $\rho$ solves~\eqref{eq:weakdd}.
\end{theorem}
\begin{proof}
From the definition of $(\{\rho^\delta_i\}\{g^\delta_{ij}\})$ we have that 
\[
\mathcal{J}_{\delta}(\{\rho^\delta_i\}\{g^\delta_{ij}\})\le \mathcal{J}_{\delta}(\{\rho_i\}\{g_{ij}\}),
\]
which is,
\[
\sum\limits_{ij}\left[\int_{S_{ij}}(\rho^\delta_i-\rho^\delta_j)^2 d S_{ij} +\frac{\delta}{2}\int_{S_{ij}} (g_{ij}^\delta)^2 \,\de S_{ij}\right] \le \frac{\delta}{2}\int_{S_{ij}} g_{ij}^2 \,\de S_{ij},
\]
implying, by the uniform boundedness of $g_{ij}$ in $L^2(S_{ij})$ that $\|\rho^\delta_i-\rho^\delta_j\|_{L^2(S_{ij})}\to 0$ as $\delta\to 0$. Furthermore, from~\eqref{eq:energyesttwo} we get that $\|\rho_i^{\delta}\|_{H^1(R_i)}$ are uniformly bounded. Thus, with $\delta\to 0$ there is a subsequence converging to $\{\rho_{ij}^*,g_{ij}^*\}$ over $\{H^1(R_i),L^2(S_{ij})\}$ and passing to the limit implies they satisfy~\eqref{eq:weakpdeoptcont}. Furthermore $\|\rho^\delta_i-\rho^\delta_j\|_{H^1(S_{ij})}\to 0$ implies that $\rho_i^*\vert_{S_{ij}}=\rho_j\vert_{S_{ij}}$. Defining $\rho\in H^1(\mathbb{R}^n)$ by $\rho\vert_{R_i\cup(\cup_j \{S_{ij}\})}=\rho_i$ we obtain the unique solution to~\eqref{eq:weakdd}.
\end{proof}

We have now proven the existence of $\rho\in H^1(\mathbb{R}^n)$ satisfying the weak form of the PDE~\eqref{eq:fp} over $\cup_i R_i$, corresponding to where the coefficients are smooth almost everywhere, and with the boundary conditions~\eqref{eq:boundaryfp}.

\subsection{Variational Flow for the Free Energy}
\textcolor{black}{This section will extend the classical JKO study relating the Langevin equation to minimization of free energy to the SDI~\eqref{eq:langSDE}.}
In~\cite{jordanKinderlehrerOtto1998}, the FP equation~\eqref{eq:fp} is shown to be the gradient flow of the functional,
\begin{equation}\label{eq:var}
    \mathcal{F}(\rho) = E(\rho)+S(\rho;\sigma) = \int f(x) \rho \,\de x+\int \rho\log\rho \,\de x
\end{equation}
when $F(x)=\nabla f(x)$ and the stationary distribution is given by its minimizer. To this effect, one can consider a scheme,
\begin{equation}\label{eq:minschemefp}
    \rho^{(k)}:=\mathop{\arg\min}\limits_\rho \frac{1}{2} W(\rho^{(k-1)},\rho)^2+h \mathcal{F}(\rho)
\end{equation}
for some small $h$, where $W$ refers to the Wasserstein distance. 
\textcolor{black}{We shall see that with the arguments relying on properties of functions that are valid for weak sense solutions, we can extend the arguments readily to the case of set-valued drift, as the point-wise discontinuities do not affect the overall operation of the said functionals. As such, the stationary distribution is still $\sim e^{-f(x)}$ as expected, showing ergodicity to the expected distribution, opening up the possibility of algorithmic analysis.}

Let $M$ be defined as
\[
M\udef\left\{\rho:\mathbb{R}^n\to[0,\infty) \text{ measurable, and }\int \rho(x)\,\de x=1,\, \int\abs{x}^2 \rho(x)\,\de x<\infty\right\}
\]

We have the following Proposition, whose proof is unchanged in our setting,
\begin{proposition}\cite[Proposition 4.1]{jordanKinderlehrerOtto1998}
Given $\rho\in M$, there exists a unique solution to~\eqref{eq:minschemefp}.
\end{proposition}

Now we extend the classical main result, Theorem 5.1 in~\cite{jordanKinderlehrerOtto1998}, to our setting, which requires a few modifications to account for the nonsmooth potential in the free energy.
\begin{theorem}\label{th:fpconv}
Let $\rho^0\in M$ with $F(\rho^0)<\infty$ and $\rho^{(k)}_h$ the solution for~\eqref{eq:minschemefp} and
the interpolation $\rho_h:(0,\infty)\times\mathbb{R}^n\to [0,\infty)$ by,
\[
\rho_h(t)=\rho_h^{(k)}\text{ for }
t\in[kh,(k+1)h)\text{ and }k\in\mathbb{N}\cup\{0\}
\]
Then as $h\to 0$, $\rho_h(t)\to\rho(t)$ weakly in $L^1(\mathbb{R}^n)$ for all $t\in(0,\infty)$ where $\rho\in H^1((0,\infty)\times \mathbb{R}^n)$ \textcolor{black}{is a weak solution to}~\eqref{eq:solutionfp} with initial condition $\rho(t)\to\rho^0$ strongly in $L^1(\mathbb{R}^n)$ for $t\to 0$.
\end{theorem}
\begin{proof}
Let $\xi\in C^\infty_0(\mathbb{R}^n,\mathbb{R}^n)$ be a smooth vector field with bounded support, and its flux $\Phi_\tau$ as $\partial_\tau \Phi_\tau = \xi \circ \Phi_\tau$ for all $\tau\in\mathbb{R}$ and $\Phi_0=\mathop{id}$. The measure $\rho_\tau(y)\,\de y$ is the push forward of $\rho^{(k)}(y)\,\de y$ under $\Phi_\tau$. This means that
\[
\int_{\mathbb{R}^n} \rho_\tau(y)\zeta(y) \,\de y = \int_{\mathbb{R}^n} \rho^{(k)}(y)\zeta(\Phi_\tau(y))\,\de y
\]
for all $\zeta\in C^0_0(\mathbb{R}^n)$, which implies that $\det \nabla \Phi_\tau \rho_t\circ \Phi_\tau = \rho^{(k)}$. By the properties of~\eqref{eq:minschemefp} we have that
\begin{equation}\label{eq:usingminfp}
\frac{1}{\tau}\left[\left(\frac 12 W(\rho^{(k-1)},\rho_\tau)^2+h \mathcal{F}(\rho_\tau)\right)-
\left(\frac 12 W(\rho^{(k-1)},\rho^{(k)})^2+h \mathcal{F}(\rho^{(k)})\right)\right] \ge 0.
\end{equation}
Now, consider $\zeta = f$ and so
\[
\int_{\mathbb{R}^n} \rho_\tau(y)f(y) \,\de y = \int_{\mathbb{R}^n} \rho^{(k)}(y)f(\Phi_\tau(y))\,\de y
\]
and also
\[
\frac{1}{\tau}\left(E(\rho_\tau)-E(\rho^{(k)})\right) = 
\int_{\mathbb{R}^n} \frac{1}{\tau} \left(f(\Phi_{\tau}(y))-f(y)\right)\rho^{(k)}(y)\,\de y
\]
Recalling the consideration of a conservative vector field, cf. Definition \ref{def:conservativeMapForF} above. In the original,
the term $\nabla f(y)\cdot \xi(y)$ appears, instead, we have a vector $\chi(y)$ that represents directional change in $f$.
Specifically, we can write that
\[
\frac{d}{d\tau}[E(\rho_\tau)]_{\tau=0} = \int_{\mathbb{R}^n} \chi(y)\cdot \xi(y)\rho^{(k)}(y)\,\de y.
\]
We can continue similarly as in the proof of~\cite[Theorem 5.1]{jordanKinderlehrerOtto1998} to obtain,
\[
\frac{d}{d\tau} [S(\rho_\tau)]_{\tau=0} = -\int_{\mathbb{R}^n} \rho^{(k)} \text{div} \xi\,\de y
\]
and subsequently the following a priori estimates, wherein for any $T<\infty$ there exists $C$ such that for all $N\in\mathbb{N}$ and $h\in[0,T]$ with $Nh\le T$ there holds,
\begin{equation}\label{eq:aprioribound}
    \begin{split}
    M(\rho_h^{(N)})\le C,\\
    \int_{\mathbb{R}^n}\max\{\rho_h^{(N)}\log \rho_h^{(N)},0\}\,\de x \le C,\\
    E(\rho^{(N)}_h)\le C,\\
    \sum\limits_{k=1}^N W(\rho_h^{(k-1)},\rho_h^{(k)})^2\le Ch
    \end{split}
\end{equation}
which implies the existence of a convergent subsequence,
\[
\rho_h \rightharpoonup \rho \text{ weakly in }L^1((0,T)\times\mathbb{R}^n)\text{ for all } T<\infty
\]
with $\rho(t)\in M$ and $E(\rho)\in L^\infty((0,T))$, and $\rho$ satisfying
\begin{equation}\label{eq:weakrho}
    -\int_{(0,\infty)\times\mathbb{R}^n} \rho (\partial_t \zeta-\chi \cdot \nabla \zeta+\triangle \zeta) \,\de x\,\de t = \int_{\mathbb{R}^n} \rho^0 \zeta(0)\,\de x \text{ for all }\zeta\in C^{\infty}_0(\mathbb{R}\times\mathbb{R}^n)
\end{equation}
We can proceed similarly as in the proof of Theorem 5.1 in~\cite{jordanKinderlehrerOtto1998} using appropriate test functions to conclude that $\rho\in L^p_{loc}((0,\infty),\mathbb{R}^n)$. Now, clearly this $\rho$ solves~\eqref{eq:weakdd}, which has a unique solution in $H^1(\mathbb{R}^n, (0,T))$ by Theorem~\ref{th:limitoptcondd}. 

The final statements of the Theorem follow as in Theorem 5.1 in~\cite{jordanKinderlehrerOtto1998}.
\end{proof}



\section{Numerical Illustration}\label{sec:num}
\subsection{One-Dimensional Example}

For the first illustration, we took a one-dimensional function
\begin{equation}\label{eq:example}
f(x):=\left\{ \begin{array}{lr}
-x-1 & x< -1\\
x+1 & -1\le x < 0 \\
1-x & 0\le x < 1 \\
x-1 & x\ge 1
\end{array}\right. ,\,\,\,
F(x):=\left\{ 
\begin{array}{lr}
-1 & x< -1 \\
\left[-1,1\right] & x=-1 \\
1 & -1<x< 0 \\
\left[-1,1\right] & x= 0 \\
-1 & 0<x< 1 \\
\left[-1,1\right] & x= 1 \\
1 & x>1
\end{array}\right.
\end{equation}
The probability density function associated with minimizing the free energy is shown in Figure~\ref{fig:examplepdf} and the result of one hundred thousand samples generated by the Metropolis Algorithm in Figure~\ref{fig:metro}.
\begin{center}
    \begin{figure}
        \centering
        \includegraphics[scale=0.5]{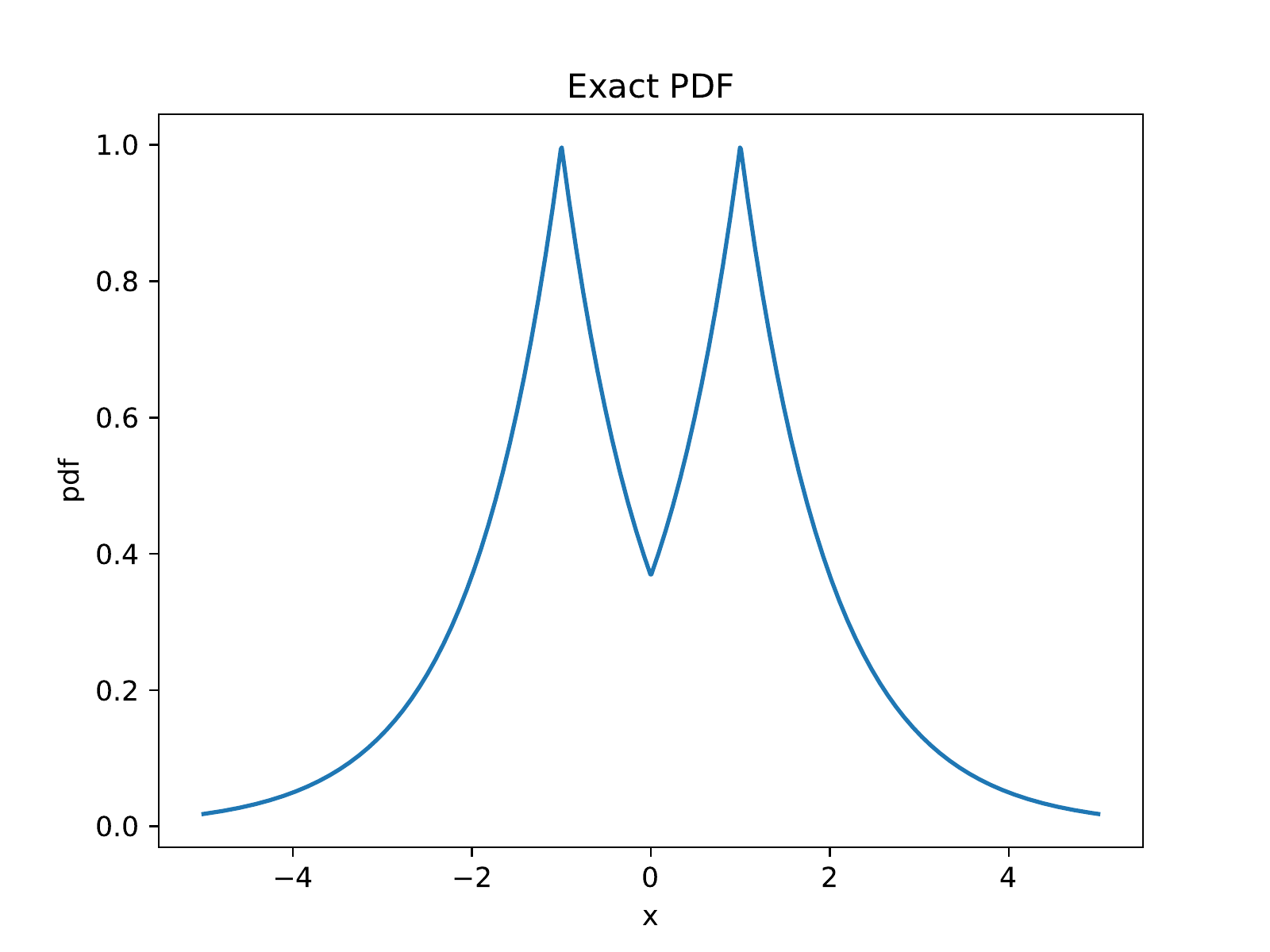}
        \caption{Stationary Distribution $e^{-f(x)}$ Associated with~\eqref{eq:example}.}
        \label{fig:examplepdf}
    \end{figure}
\end{center}

\begin{center}
    \begin{figure}
        \centering
        \includegraphics[scale=0.5]{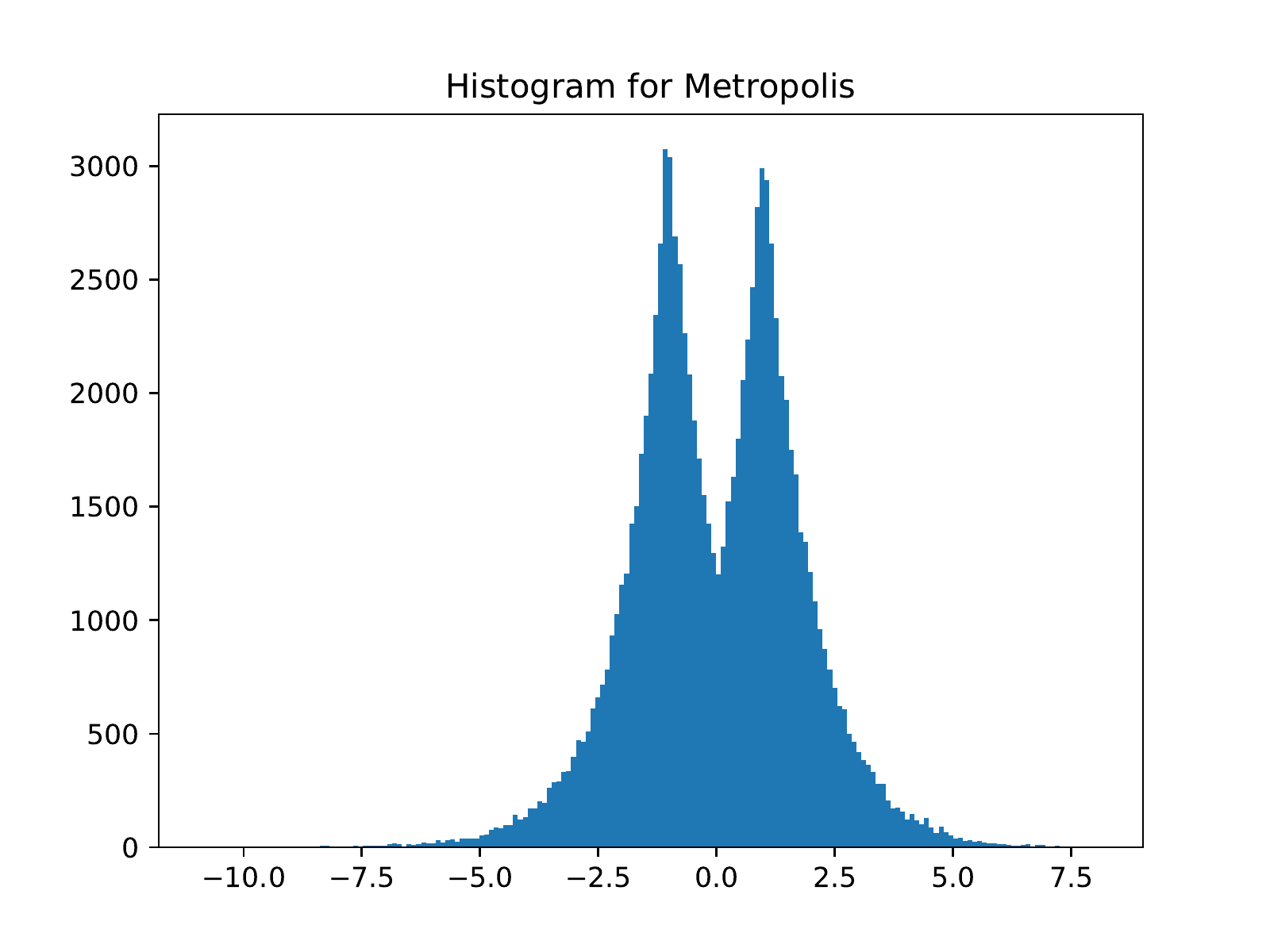}
        \caption{Histogram of 100k Metropolis Samples of~\eqref{eq:example}}
        \label{fig:metro}
    \end{figure}
\end{center}

As an illustration, we ran unadjusted Langevin dynamics with the Euler-Maruyama discretization, i.e., generated samples with the following iteration,
\begin{equation}\label{eq:iterationexample}
x_{k+1} = x_k - \epsilon g_k+\sqrt{2\epsilon } B_k,\, g\in F(x_k),\, B_k\sim N(0,1).
\end{equation}
We generated ten million samples with $\epsilon\in\{0.01,0.001,0.0001\}$. We plot the histograms of the final sample count in Figure~\ref{fig:examplehist} and plot the Wasserstein distance (computed with \texttt{scipy}) in Figure~\ref{fig:wasdist}.

\begin{center}
    \begin{figure}
        \centering
        \includegraphics[scale=0.26,trim={0 0 0 1.35cm },clip]{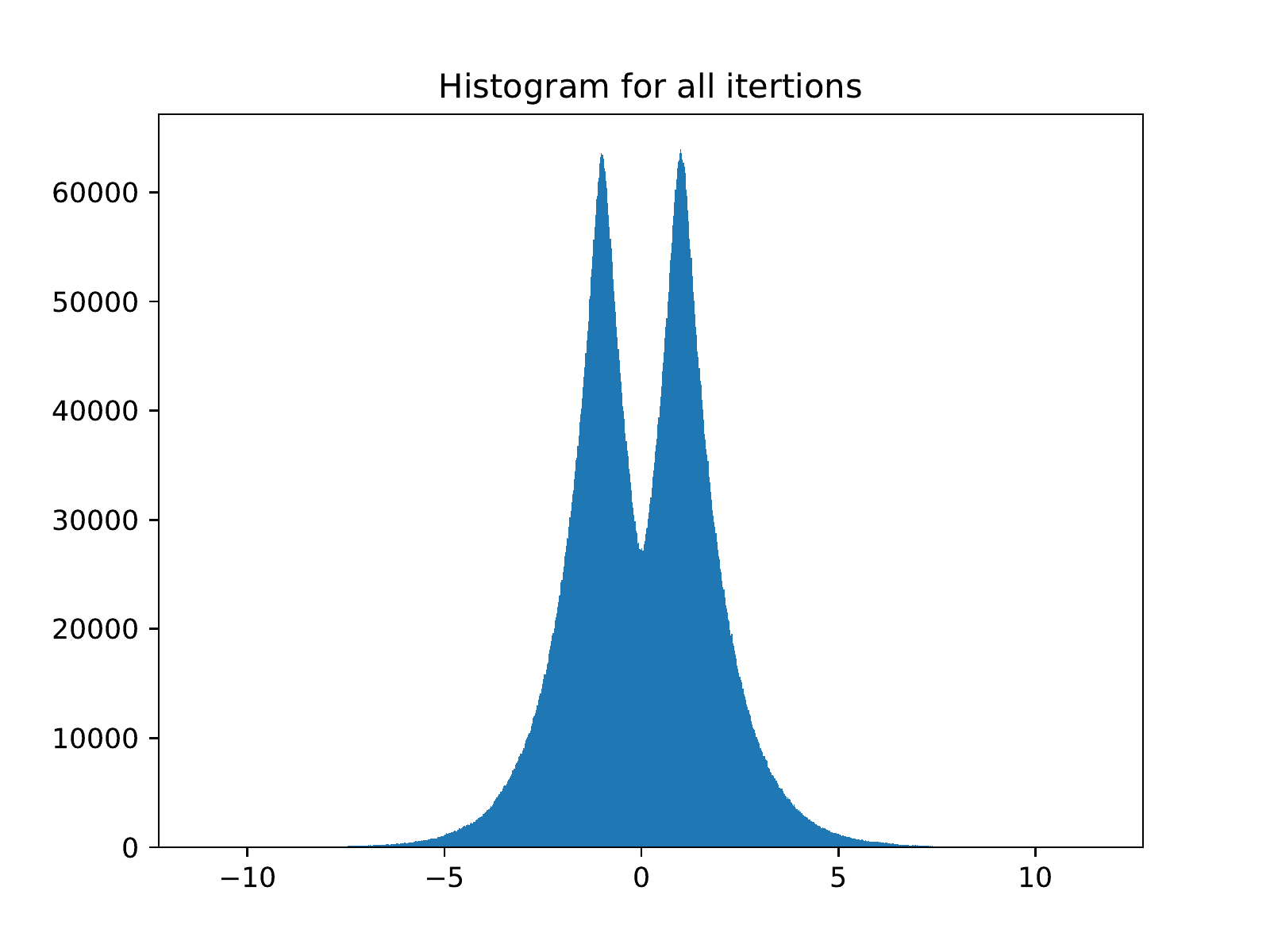}
        \includegraphics[scale=0.26,trim={0 0 0 1.35cm },clip]{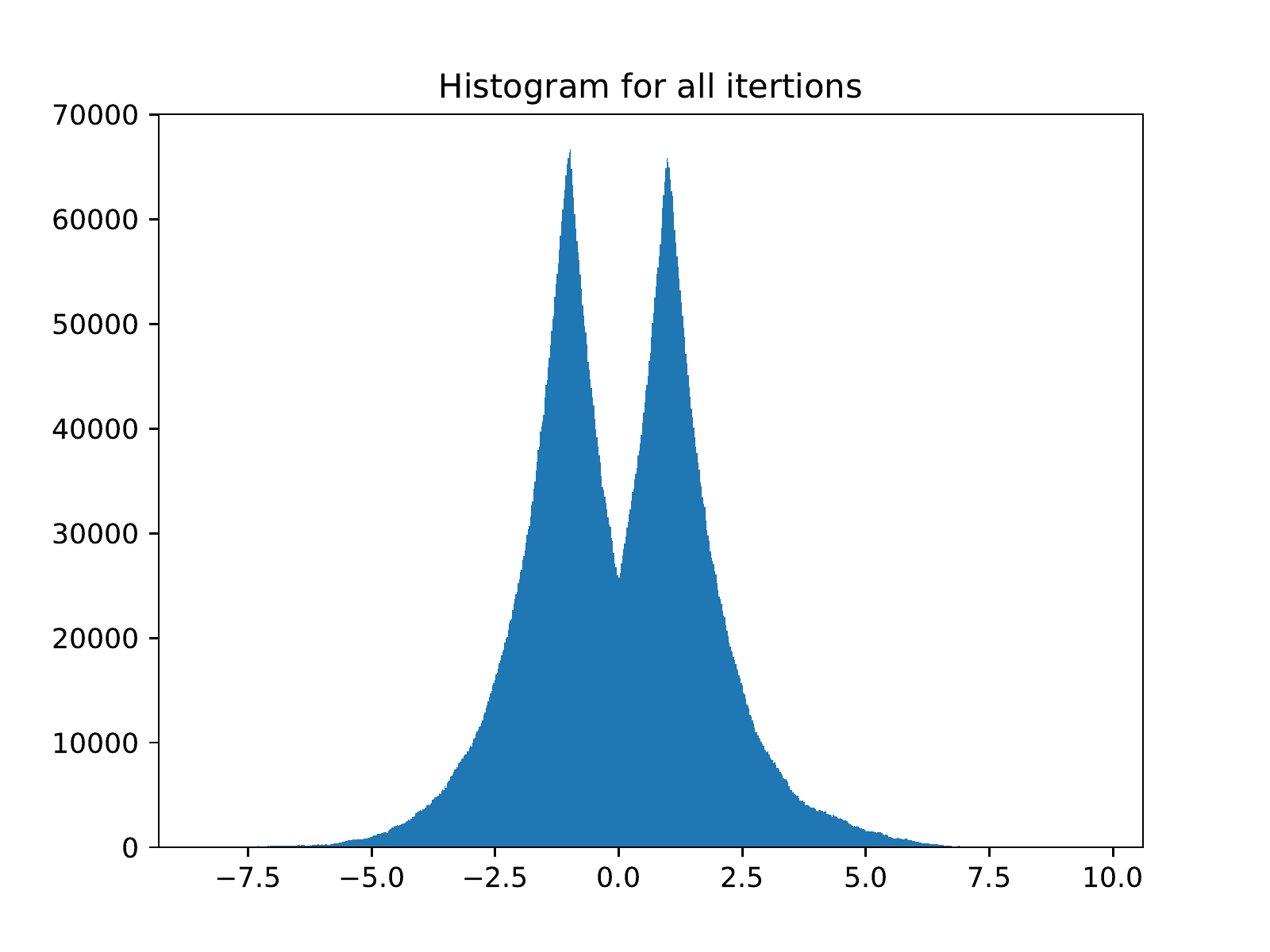}
        \includegraphics[scale=0.26,trim={0 0 0 1.35cm },clip]{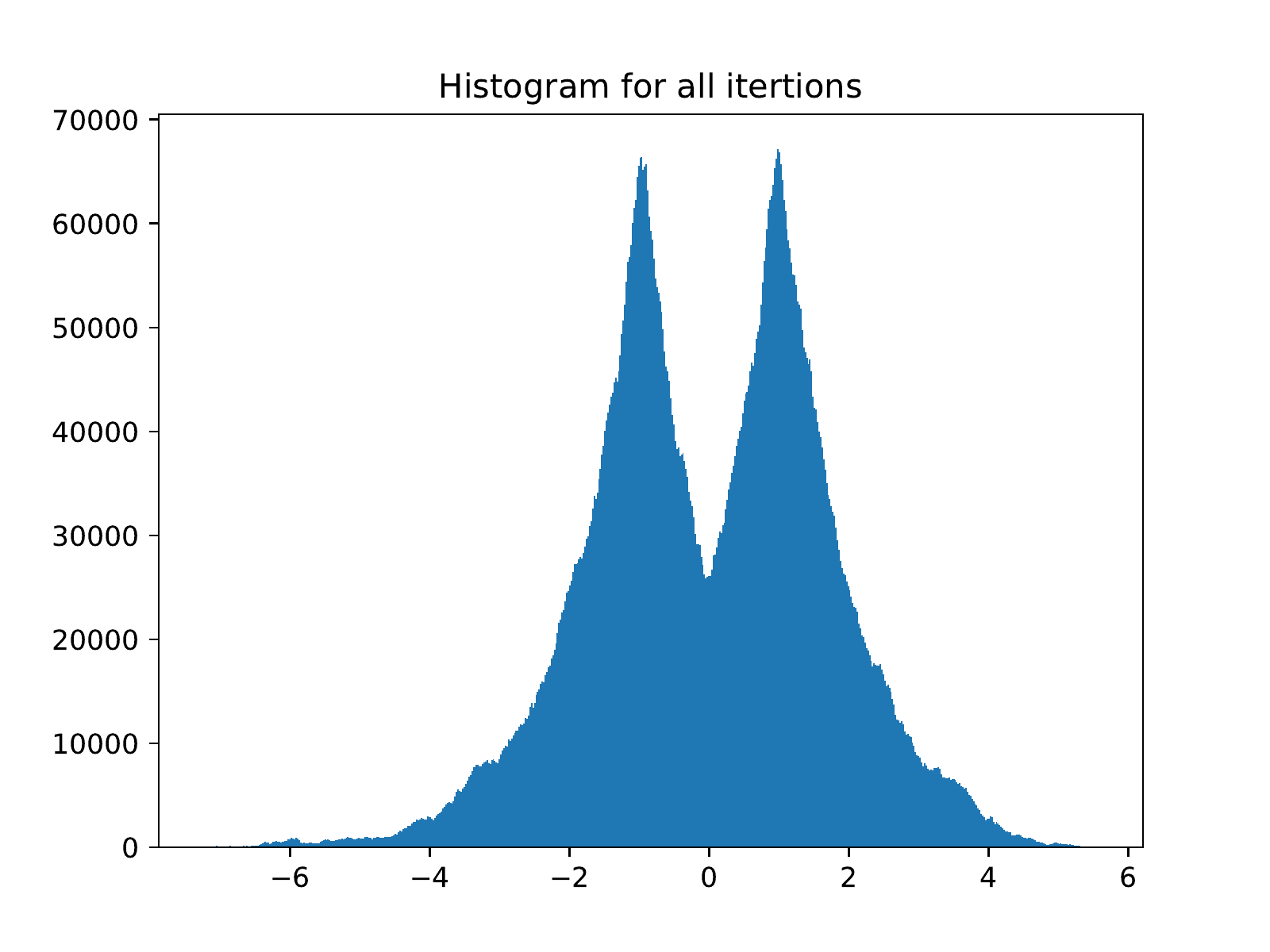}
        \caption{Histogram for~\eqref{eq:iterationexample} at $\epsilon\in\{0.01,0.001,0.0001\}$ in this order. 
        }
        \label{fig:examplehist}
    \end{figure}
\end{center}

\begin{center}
    \begin{figure}
        \centering
        \includegraphics[scale=0.26]{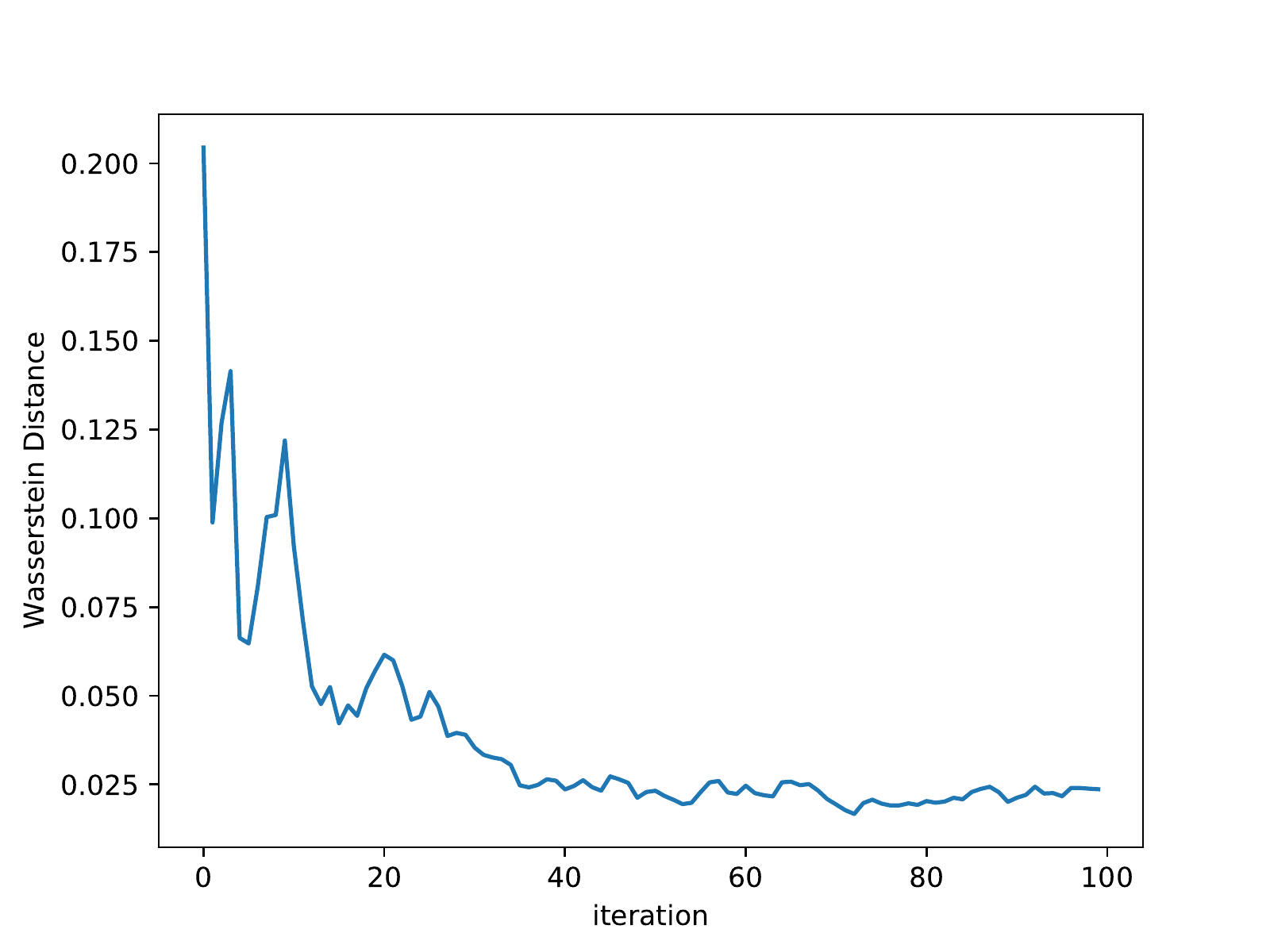}
        \includegraphics[scale=0.26]{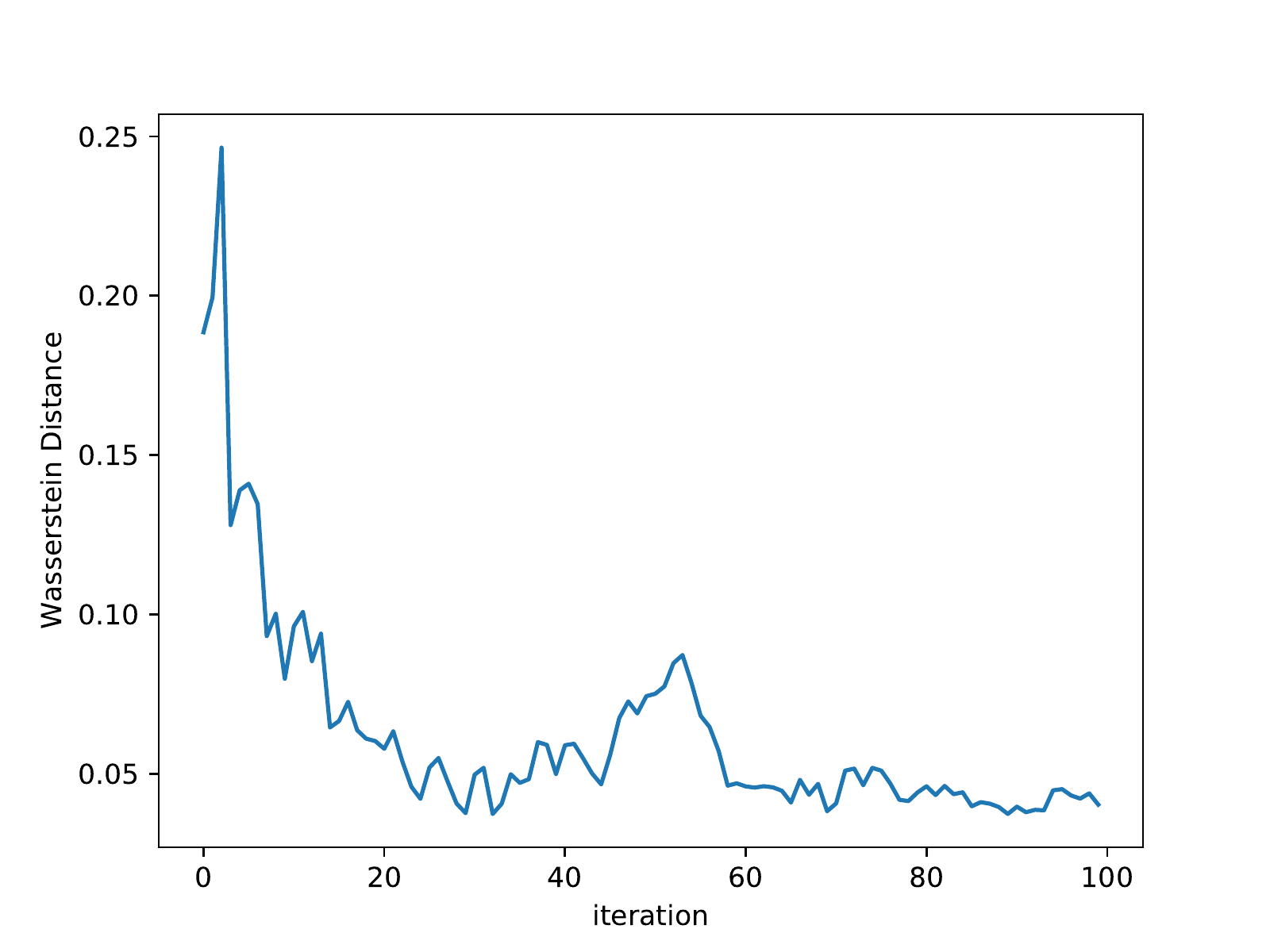}
        \includegraphics[scale=0.26]{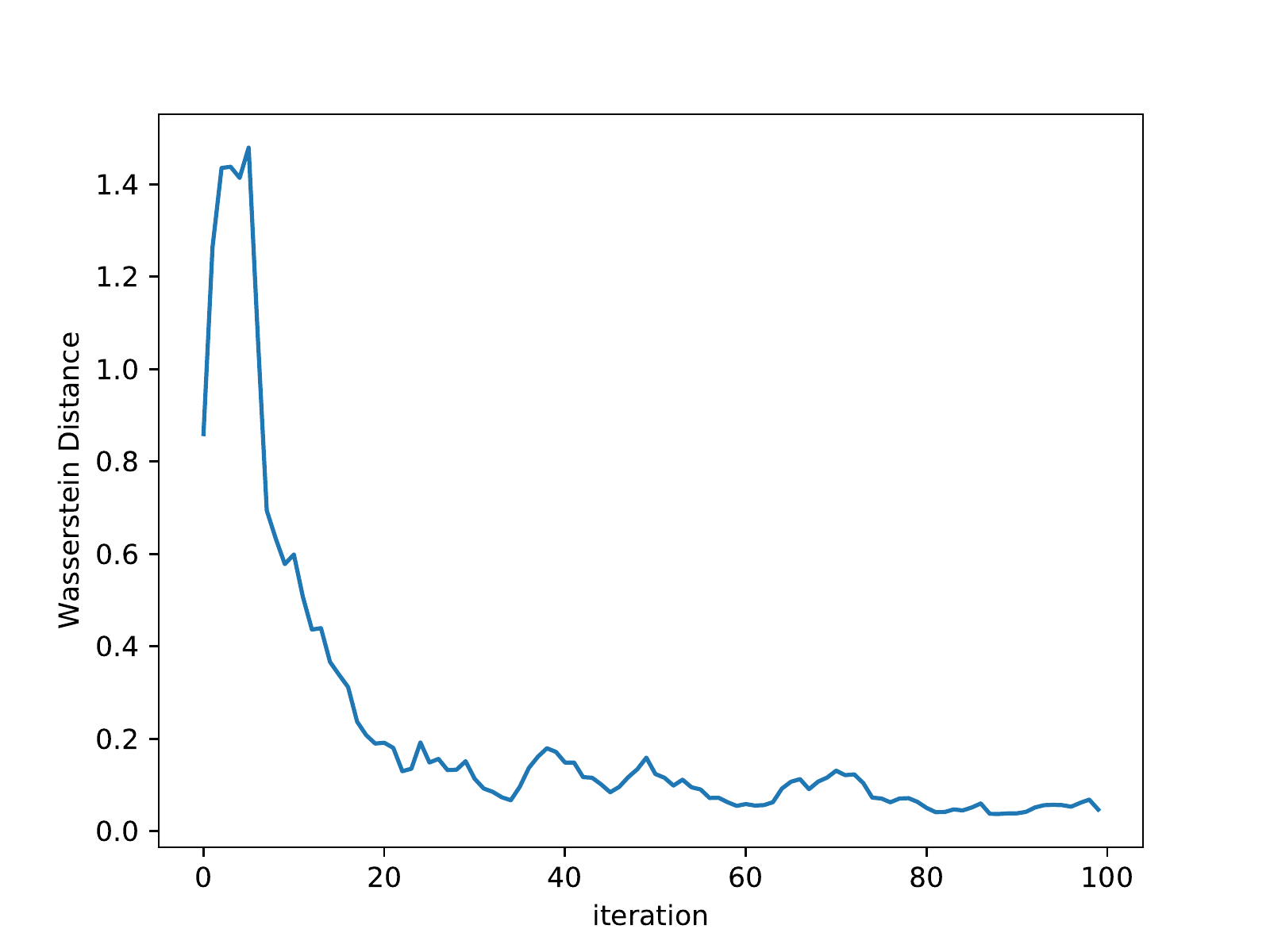}
        \caption{Wasserstein Distance of the Langevin-type iteration~\eqref{eq:iterationexample} and Metropolis-generated samples for $\epsilon\in\{0.01,0.001,0.0001\}$, in this order.}
        \label{fig:wasdist}
    \end{figure}
\end{center}
We can observe that indeed the iteration does recreate the posterior, i.e., it appears to be ergodic, although the quantity of samples required is fairly large, indeed it seems to converge in probability distance after about 3 million samples, suggesting geometric ergodicity is unlikely.

\subsection{Bayesian ReLU Neural Network}
The use of gradient-based samplers had been introduced due to their improved mixing rate with respect to dimensionality dependence. Although we do not derive quantitative mixing rates in this work, it would be naturally suspected that such behavior could carry over to the nonsmooth case. For this purpose, we perform Metropolis-Hastings, (subgradient) unadjusted Langevin, and Metropolis-corrected Langevin on a ReLU network. To avoid complications associated with inexactness, we used moderately sized datasets and performed backpropagation on the entire data sample, rather than a stochastic variant.

Specifically, we consider the \texttt{E2006} and \texttt{YearPredictionMSD} datasets from the UCI LIBSVM datset repository~\citep{chang2011libsvm}. Both datasets have high parameter dimension, 150000 and 90, respectively, and with 16k or 460k as the number of training samples. We used a ReLU network with three hidden layers, each with 10 neurons. In this case, with the high dimension, rather than attempting to visualize the posterior, we plot the averaged (across 20 runs) of the loss on the test set. See Figure~\ref{fig:reluresult} for the results of the Langevin and Metropolis-adjusted Langevin approach on the test accuracy for dataset \texttt{E2006}. In this case, Metropolis resulted in
a completely static mean high variance test loss across the samples, being entirely uninformative. Next, Figure~\ref{fig:reluresult2} shows the test loss
for a ReLU network using unadjusted Langevin. In this case, note
the noisy initial plateau is followed by decrease. For both
Metropolis as well as the Metropolis-corrected variant of
Langevin there is no improvement due to repeated step rejection of the initial sample or minimal change in the error. 
\begin{center}
    \begin{figure}
        \centering
        \includegraphics[scale=0.42]{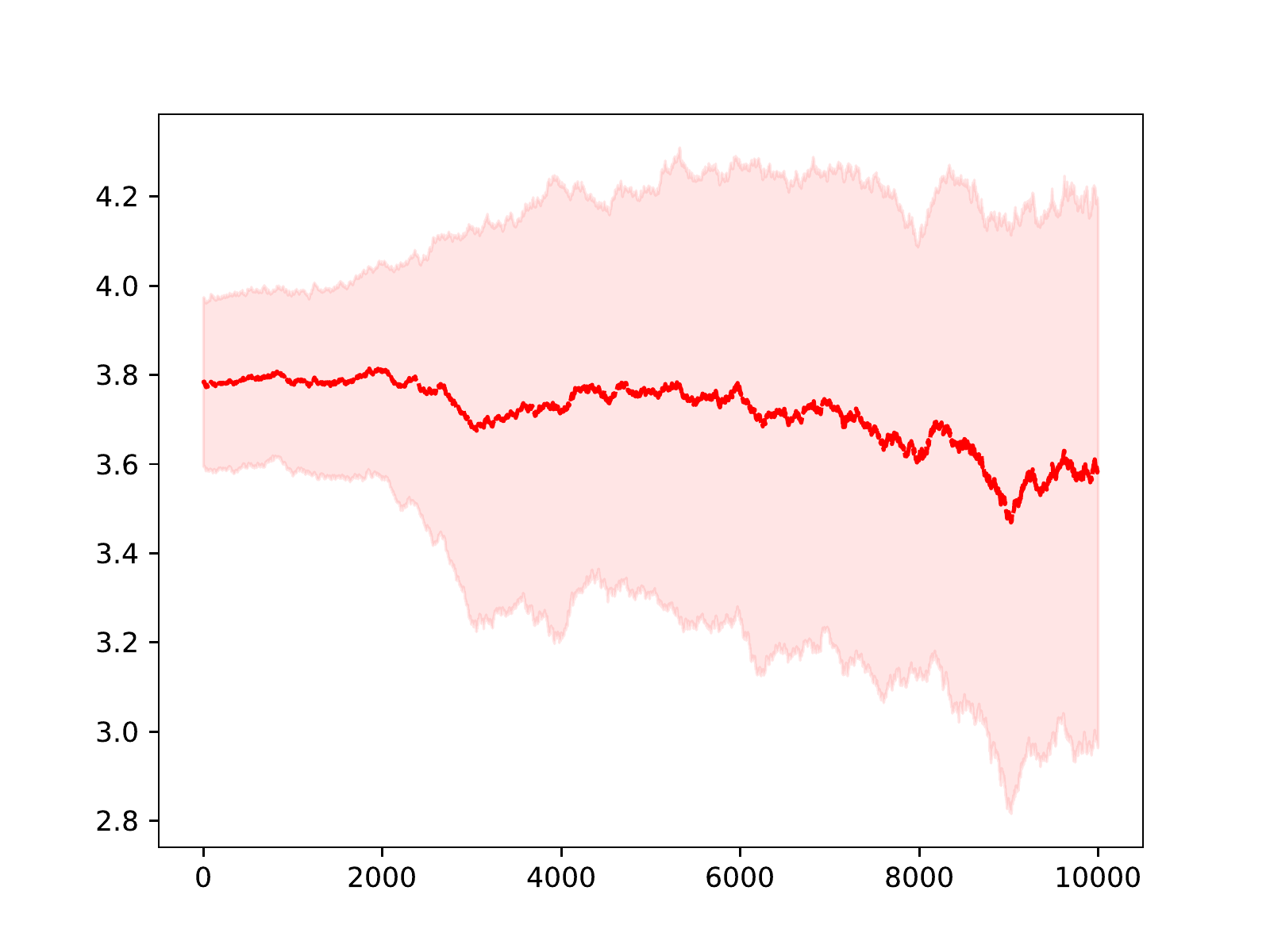}
        \includegraphics[scale=0.42]{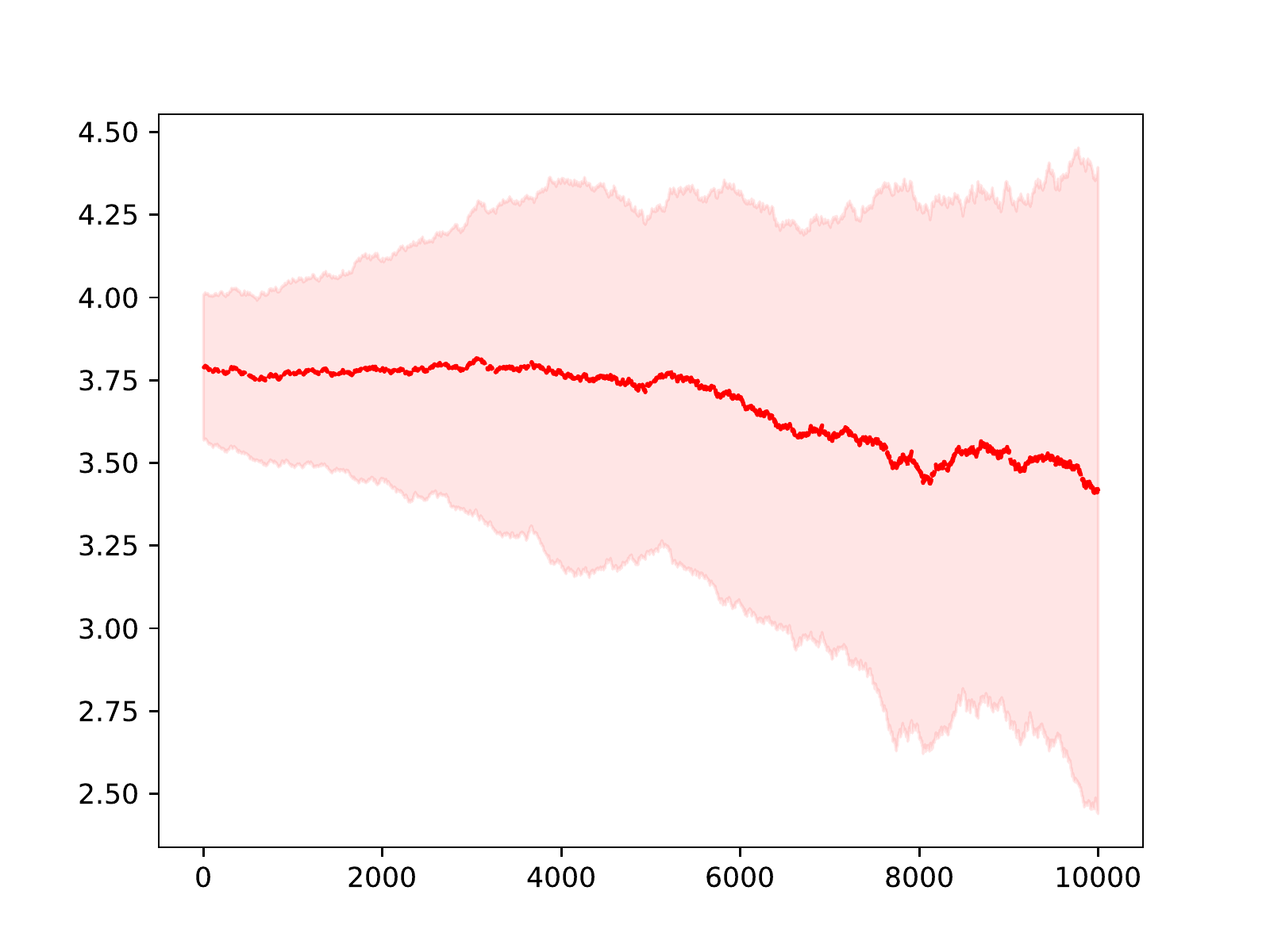}
        \caption{Test Loss on the sampled parameters generated by the unadjusted as well as the Metropolis-adjusted Langevin method on a ReLU Neural Nework on a regression task for the dataset \texttt{E2006}. The discretization rate is 1e-05.}
        \label{fig:reluresult}
    \end{figure}
\end{center}

\begin{center}
    \begin{figure}
        \centering
        \includegraphics[scale=0.8]{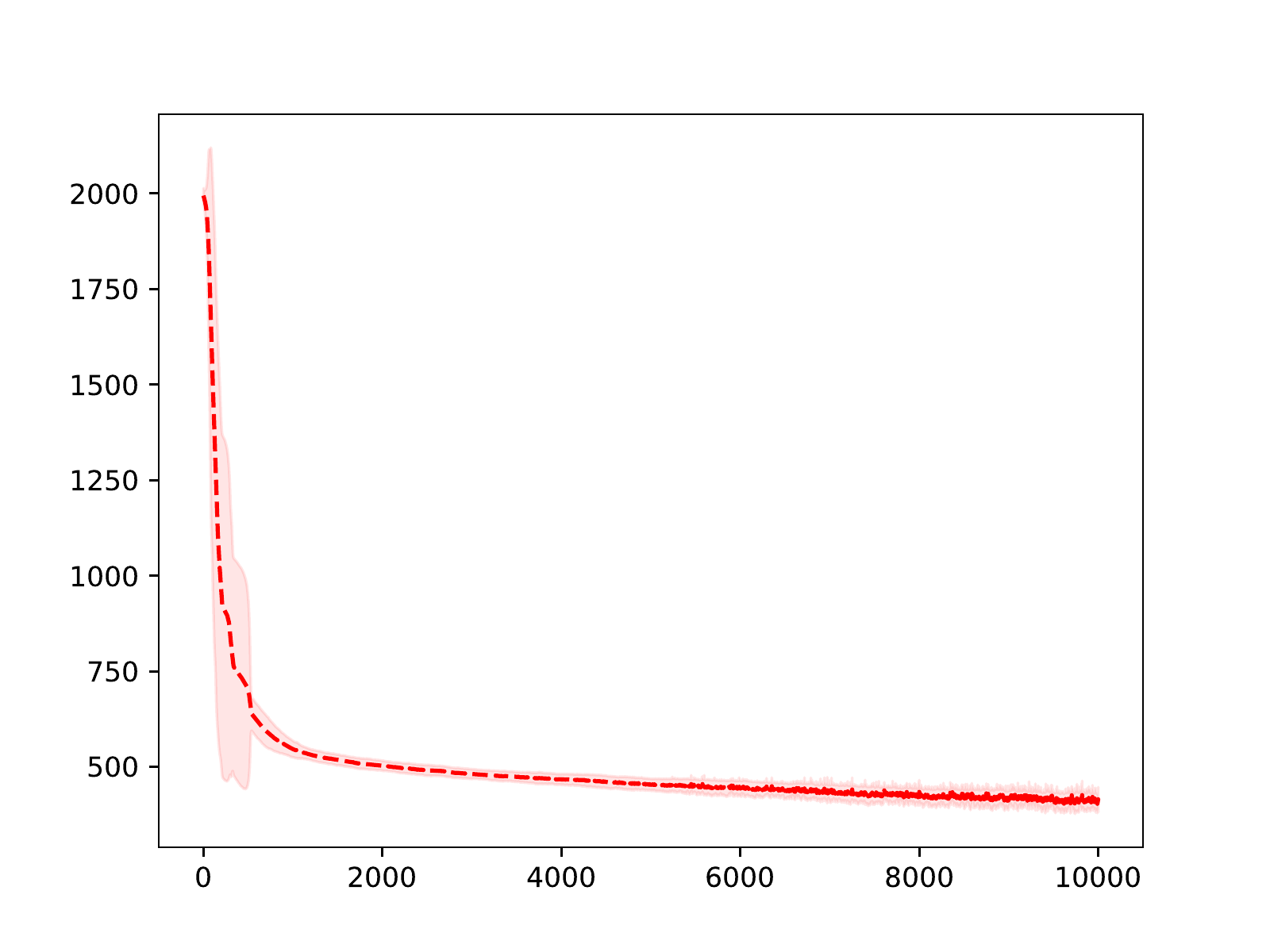}
        \caption{Test Loss on the sampled parameters generated by the unadjusted Langevin method on a ReLU Neural Nework on a regression task for the dataset \texttt{YearPredictionMSD}. The discretization rate is 1e-05.}
        \label{fig:reluresult2}
    \end{figure}
\end{center}

\section{Discussion and Implications}\label{sec:conc}
The standard potential gradient diffusion process has featured prominently in the theoretical analysis meant to give insight as to the approximation and generalization performance associated with the long-term behavior of SGD as applied to neural networks, for example,~\cite{hu2017diffusion}. It has also featured prominently in algorithms for sampling high-dimensional data sets, e.g.~\cite{bussi2007accurate}. It is standard for these studies to require that the drift term is Lipschitz and, as such, that the potential is continuously differentiable. In many contemporary applications, for example in (Bayesian, in the case of sampling) Neural Networks with ReLU or convolutional layers, this is not the case, and the presence of points wherein the loss function is not continuously differentiable is endemic. Therefore, the primary results of this paper, the existence of a solution to the stochastic differential inclusion drift as well as the existence of a Fokker-Planck equation characterizing the evolution of the probability distribution which asymptotically converges to a Gibbs distribution of the free energy, provide some basic insights into these processes without unrealistic assumptions. Specifically, even with these nonsmooth elements as typically arising in Whitney-stratifiable compositions of model and loss criteria, the overall understanding of the asymptotic macro behavior of the algorithmic processes remains as expected.

Additional insights gathered from studying an approximating SDE are not as straightforward to extend to the differential-inclusion setting. 
The issues of wide and shallow basins around minima seem minor, when one considers that a Hessian of a loss function may not exist at certain points. Thus, considerations of mixing rate for sampling and qualitative properties of limiting distributions are interesting topics for further study in specific cases of specific stochastic differential inclusions.



\acks{
We'd like to thank Thomas Surowiec for helpful discussions regarding PDE theory as related to the arguments in this paper. FVD has been supported by \textit{REFIN} Project, grant number 812E4967 funded by Regione Puglia; he is also part of the INdAM-GNCS research group. VK and JM were supported by the OP VVV project CZ.02.1.01/0.0/0.0/16\_019/0000765 ``Research
Center for Informatics", 
by the European Union’s Horizon Europe research and innovation programme under grant agreement No. GA 101070568 (Human-compatible AI with guarantees),
and by the
Czech Science Foundation (22-15524S).
}




\newpage








\vskip 0.2in
\bibliography{refs}

\end{document}